\documentclass[12pt]{amsart}
\usepackage{latexsym}
\usepackage{amssymb, amsmath,mathrsfs,amsthm}
\usepackage{geometry}
\usepackage{setspace}
\usepackage{mathtools}
\usepackage{graphicx, subfigure}
\usepackage{tikz}
\usepackage{pgfplots}

\usepackage{enumerate}
\usepackage{float}
\usepackage[pagebackref,hypertexnames=false, colorlinks, citecolor=red, linkcolor=red]{hyperref}

\newcommand{\McC}{\raise.5ex\hbox{c}}

\newtheorem{theorem}{Theorem}[section]
\newtheorem*{theorem*}{Theorem}
\newtheorem{lemma}[theorem]{Lemma}
\newtheorem{definition}[theorem]{Definition}

\newtheorem{corollary}[theorem]{Corollary}
\newtheorem*{corollary*}{Corollary}
\newtheorem{proposition}[theorem]{Proposition}

\usepackage{color}

\makeatletter
\@namedef{subjclassname@2020}{%
  \textup{2020} Mathematics Subject Classification}
\makeatother
\theoremstyle{remark}
\newtheorem{remark}[theorem]{Remark}
\newtheorem{example}[theorem]{Example}

\author[Bickel]{Kelly Bickel$^\dagger$}
\address{Department of Mathematics, Bucknell University, Lewisburg, PA 17837, USA.}
\email{kelly.bickel@bucknell.edu}
\thanks{$\dagger$ Research supported in part by National Science Foundation
DMS grant \#2000088.}

\author[Cima]{Joseph A. Cima}
\address{Department of Mathematics, University of North Carolina, Chapel Hill, NC 27599 USA.}
\email{cima@email.unc.edu}

\author[Sola]{Alan A. Sola}
\address{Department of Mathematics, Stockholm University, 106 91 Stockholm, Sweden.}
\email{sola@math.su.se}

\keywords{Clark measure, rational inner function, unitary embedding}
 
  \subjclass[2020]{Primary 28A25, 28A35; Secondary 32A08, 47A55}

\begin{document}
\title[Clark measures for RIFs]{Clark measures for rational inner functions}
\date{\today}

\maketitle
\begin{abstract}
We analyze the fine structure of Clark measures and Clark isometries associated with two-variable rational inner functions on the bidisk. In the degree $(n,1)$ case, we give a complete description of  supports and weights for both generic and exceptional Clark measures, characterize when the associated embedding operators are unitary, and give a formula for those embedding operators. We also highlight connections between our results and both the structure of Agler decompositions and study of extreme points for the set of positive pluriharmonic measures on 2-torus. 
\end{abstract}

\section{Introduction}

A bounded analytic function $\phi\colon \mathbb{D}^d\to \mathbb{C}$ is said to be {\it inner} if $|\phi(\zeta)|=1$ for almost every $\zeta \in \mathbb{T}^d$, where $\mathbb{D}$ is the unit disk and $\mathbb{T}$ is the unit circle. In the one-variable case, each inner function $\psi$ defines a class of positive Borel measures $\{ \sigma_\alpha\}_{\alpha \in \mathbb{T}} $ on $\mathbb{T}$ that satisfy
\[ \frac{1-|\psi(z)|^2}{|\alpha-\psi(z)|^2} = \int_\mathbb{T} \frac{ 1-|z|^2}{|\zeta - z|^2} d\sigma_\alpha(\zeta), \quad \text{ for } z \in \mathbb{D}.\]
These measures have a number of important applications and properties; among other results, they are the spectral representing measures for rank $1$ unitary perturbations of certain compressed shift operators and via Alexandrov's theorem, they disintegrate Lebsegue measure, see  \cite{CMR,GMR} for comprehensive introductions to this classical theory. Generalizations of these measures to the polydisk $\mathbb{D}^d$ were recently studied by E. Doubtsov in \cite{D19}; other multivariate generalizations of Clark theory can be found in \cite{AD20, J14}. In this paper, we obtain precise information about both the two-variable Clark measures on the bidisk defined in  \cite{D19} and associated isometries, in the setting of two-variable rational inner functions. 

\subsection{Notation and Setup} 
To define Clark measures on the bidisk, we need some notation. Denote the Poisson kernel on $\mathbb{D}^2$ by 
\[ P_z(\zeta)= P(z, \zeta) := \frac{ (1-|z_1|^2)(1-|z_2|^2)} {|\zeta_1-z_1|^2 |\zeta_2 -z_2|^2}, \quad \text{ for } z\in \mathbb{D}^2, \zeta \in \mathbb{T}^2 ,\] 
and the Cauchy kernel for the bidisk by
\[C_w(z) = C(z, w)=\frac{1}{(1-z_1\overline{w}_1)(1-z_2\overline{w}_2)},\quad \text{ for } z\in \mathbb{D}^2, w \in \overline{\mathbb{D}}^2.\]
Recall that $C$ acts as the reproducing kernel for the Hardy space $H^2(\mathbb{D}^2)$, which consists of all analytic functions $f\colon \mathbb{D}^2\to \mathbb{C}$ satisfying the norm boundedness condition
\[\|f\|^2_{H^2}:=\sup_{0<r<1}\int_{\mathbb{T}^2}|f(r\zeta)|^2dm_2(\zeta)<\infty.\]
Here, $m_2$ denotes normalized Lebesgue measure on $\mathbb{T}^2$ and later, we will use $m$ to denote normalized Lebesgue measure on $\mathbb{T}$. For $f \in H^2(\mathbb{D}^2)$ and $\zeta \in \mathbb{T}^2$, we let $f^*(\zeta)$ denote the non-tangential value of $f$ at $\zeta$. Recall that $f^*(\zeta)$ exists for a.e. $\zeta \in \mathbb{T}^2$, see Chapter XVII, Theorem 4.8 in \cite{Zygmund}.  Also, throughout this paper, we will slightly abuse notation by using $z, \zeta$ to refer to points from both $\mathbb{C}$ and $\mathbb{C}^2$, but the meaning should be clear from the context.

Let $\phi$ be a non-constant inner function on $\mathbb{D}^2$ and let $\alpha \in \mathbb{T}$. Since $z\mapsto \Re[(\alpha+\phi(z))/(\alpha-\phi(z))]$ is a positive pluriharmonic function on the bidisk, there exists a unique positive Borel measure $\sigma_\alpha$ on $\mathbb{T}^2$ called a \emph{Clark measure} such that 
\[ \Re \left( \frac{\alpha + \phi(z) }{\alpha - \phi(z)} \right) = \frac{1-|\phi(z)|^2}{|\alpha-\phi(z)|^2}= \int_{\mathbb{T}^2} P_z(\zeta) d \sigma_{\alpha}(\zeta),\]
for all $z \in \mathbb{D}^2.$ Observe that 
\[ \int_{\mathbb{T}^2} d \sigma_{\alpha}(\zeta) =\int_{\mathbb{T}^2} P(0,\zeta) d \sigma_{\alpha}(\zeta) = \Re \left( \frac{\alpha + \phi(0) }{\alpha - \phi(0)} \right) = \frac{ 1- |\phi(0)|^2}{|\alpha - \phi(0)|^2} <\infty, \]  
so $\sigma_\alpha$ is a finite measure. Since $\sigma_\alpha$ is a finite Borel measure on $\mathbb{T}^2$, it is actually a Radon measure and basic measure theory (see for example \cite[Proposition 7.9]{Folland}) implies that $C(\mathbb{T}^2)$ is dense in $L^2(\sigma_{\alpha}).$  Furthermore, as linear combinations of the Poisson kernels $\{P_z\}_ {z\in\mathbb{D}^2}$ are dense in $C(\mathbb{T}^2)$, they are also dense in $L^2(\sigma_\alpha)$. Finally, as asserted in \cite{D19}, the support of each $\sigma_\alpha$ should be contained in the closure of the set $\{ \zeta \in \mathbb{T}^2: \phi^*(\zeta) =\alpha\}.$ This support condition can also be verified directly in the case when $\phi$ is a rational inner function. 

There are close connections between the Clark measures $\sigma_\alpha$ and the {\it model space} associated with the function $\phi$ defined by
 \[K_{\phi}:= H^2(\mathbb{D}^2) \ominus \phi H^2(\mathbb{D}^2).\]  
Then the reproducing kernel for $K_\phi$ is given by 
\[k(z,w) = k_{w}(z):=(1-\overline{\phi(w)}\phi(z))C_w(z), \text{ for } z, w \in \mathbb{D}^2.\]
 In \cite{D19}, Doubtsov defined an embedding map $J_{\alpha}\colon K_{\phi} \rightarrow L^2(\sigma_{\alpha})$ by first specifying it on reproducing kernels as
\[J_{\alpha}[k_w](\zeta):=(1-\alpha\overline{\phi(w)})C_w(\zeta), \quad \text{for }w \in \mathbb{D}^2, \zeta \in \mathbb{T}^2,\]
then showing this definition preserves inner products on  linear combinations of reproducing kernels, and finally extending it to all of $K_{\phi}$ using density.

In one variable, it is a classical fact (see \cite[Chapter 9]{CMR} or \cite[Chapter 11]{GMR}) that the analogous embedding is in fact a unitary for each inner function and each $\alpha \in \mathbb{T}$. In the higher-dimensional setting, Doubtsov \cite[Theorem 3.2]{D19} shows that $J_{\alpha}$ is a unitary operator if and only if the bidisk algebra $A(\mathbb{D}^2)$ is dense in $L^2(\sigma_{\alpha}).$ 
He also gives examples showing that the two-variable embeddings $J_{\alpha}$ can {\it fail} to be unitary. In this paper, we investigate this phenomenon as part of our detailed study of two-variable Clark measures associated with rational inner functions, i.e.~with functions that are both rational and inner on $\mathbb{D}^2$. It is worth noting that we restrict to this two-variable situtation (rather than a more general $d$-variable setting) because many of our key tools, which include Agler decompositions, certain model space properties, and well-understood unimodular level set behaviors, do not extend to even the three-variable setting, see \cite{BPS19, Kne11}.

To describe the structure of two-variable rational inner functions, or RIFs, we require some notation. For a polynomial $p \in \mathbb{C}[z_1, z_2]$, let $\deg_ip$ denote the degree of $p$ in $z_i$ and set $\deg p = (\deg_1p, \deg_2p).$ Then if a pair of nonnegative integers $m$ satifies $m =(m_1, m_2) \ge \deg p$, we can define the $m$-reflection  of $p$ as
\[ \tilde{p}(z) = z_1^{m_1} z_2^{m_2} \overline{ p\left( 1/\bar{z}_1, 1/\bar{z}_2 \right)}.\]
Rudin and Stout \cite{Rud69, RudSt65} showed that each RIF $\phi$ is of the form
\[ \phi(z) = \gamma \frac{\tilde{p}(z)}{p(z)} ,\]
where $\gamma \in \mathbb{T}$, $p$ has no zeros on $\mathbb{D}^2$, $\tilde{p}$ is some $m$-reflection of $p$, and $\tilde{p}, p$ share no common factors. We define $\deg \phi = \deg \tilde{p} \ge \deg p.$ To simplify our notation, we will assume $\gamma=1$ for the duration of the paper. It is worth noting that, in the one-variable setting, each RIF is a finite Blaschke product and extends analytically to some disk containing $\overline{\mathbb{D}}$ in its interior. 

Unlike this one-variable situation, two-variable rational inner functions can possess boundary singularities. For example,
\begin{equation} \label{eqn:ex} \phi(z) = \frac{2z_1z_2-z_1-z_2}{2-z_1-z_2}\end{equation}
has a singularity at $(1,1) \in \mathbb{T}^2$. More generally, if $\phi = \tilde{p}/p$ and $p(\tau)=0$ for some $\tau \in \mathbb{T}^2$, then $\tau$ is a  {\it singularity} of $\phi$ in the sense that $\phi$ cannot be extended continuously to a neighborhood of $\tau$; see \cite[Corollary 1.7]{Pas17}. Moreover, if $p(\tau)=0$, then $\tilde{p}(\tau)=0$ and so B\'ezout's theorem gives a bound on the number of zeros $p$ (or equivalently, the number of singularities $\phi$) can have on $\mathbb{T}^2$. Specifically, if $\deg \tilde{p} =(m_1,m_2)$ and $\deg p = (n_1,n_2)$, B\'ezout's theorem implies that 
 $p$ and $\tilde{p}$ have exactly $n_1m_2+n_2m_1$ common zeros in $\mathbb{C}_{\infty} \times \mathbb{C}_{\infty}$ counted according to intersection multiplicity, where $\mathbb{C}_\infty$ denotes the Riemann sphere; see p. 1287 in \cite{Kne15}. If all such common zeros of $\tilde{p}$ and $p$ occur on $\mathbb{T}^2$, we say $p$ is $\mathbb{T}^2$\emph{-saturated}. As the intersection multiplicity of such common zeros on $\mathbb{T}^2$ must be even, $\phi$ can have at most $m_1m_2$ distinct singularities on $\mathbb{T}^2$, see \cite{Kne15}. 

Still, these RIF singularities are somewhat mild. Indeed, if $\phi$ is a RIF, then \cite[Corollary 14.6]{Kne15} states that for each $\zeta \in \mathbb{T}^2$, including any points $\zeta$ where $p(\zeta)=0$, the non-tangential value $\phi^*(\zeta)$ exists and is unimodular. For more information about the zero set of $p$, denoted $\mathcal{Z}_p$, see \cite{AMS06, Kne15}.

\subsection{Overview of Results}

The body of this paper begins with Section \ref{sec:RIF}, which provides some information about the Clark measures $
\sigma_\alpha$ associated to a general RIF $\phi$. Specifically, Theorem \ref{thm:RIF1} gives a simple proof that $\sigma_\alpha$ cannot possess any point-masses (a fact noted earlier in \cite{MacD90}), and the section also gives further information about the closed set
\begin{equation} \label{eqn:Ca1} \mathcal{C}_{\alpha}:=\{\zeta \in \mathbb{T}^2\colon \tilde{p}(\zeta)=\alpha p(\zeta)\},\end{equation}
which contains the support of $\sigma_\alpha$. 

From Section \ref{sec:n1} onward,  we study RIFs $\phi =\frac{\tilde{p}}{p}$ with $\deg \phi =(n,1).$ In a sense, these are the simplest two-variable RIFs, but the constructions from \cite{BPS17, BPSprep} and the examples in our Section \ref{sec:examples} show that they can still be quite complicated. First, note that for these RIFs, 
 \begin{equation} \label{eqn:p1p2} p(z)= p_1(z_1) + z_2 p_2(z_1)\end{equation} is a  polynomial of degree at most $(n,1)$ that does not vanish on $\mathbb{D}^2$, 
\[ \tilde{p}(z) := z_2 \tilde{p}_1(z_1) + \tilde{p}_2(z_1), \text{ where each } \tilde{p}_i(z_1) = z_1^n\overline{ p_i(1/\bar{z_1})},\]
the polynomials $p$, $\tilde{p}$ share no common factors, and $p$ has at most $n$ distinct zeros on $\mathbb{T}^2$. In Subsection \ref{sec:model}, we recall some important properties about the model spaces  and formulas associated to such RIFs. For example, such RIFs possess a specific \emph{Agler decomposition} or \emph{sums of squares formula} of the form
\begin{equation} \label{eqn:AD2}   p(z) \overline{p(w)} - \tilde{p}(z) \overline{\tilde{p}(w)} = (1-z_1 \bar{w}_1) \sum_{j=1}^n R_j(z) \overline{R_j(w)}+ (1-z_2\bar{w}_2) Q(z) \overline{Q(w)}  \end{equation} 
where $R_1, \dots, R_n, Q \in \mathbb{C}[z_1, z_2]$, $\deg R_j \le (n-1,1),$ and $\deg Q \le (n,0).$

In Subsection \ref{sec:support}, we study some preliminary objects, which are key in analyzing both the Clark measures $\sigma_\alpha$ and the isometric operators $J_\alpha$ associated to $\phi$. Those objects are detailed in the following definition:

 \begin{definition} \label{def:1} Fix $\phi = \tilde{p}/p$ with $\deg \phi= (n,1)$ and $\alpha \in \mathbb{T}.$ Define the following:\\
 
\begin{itemize}
\item The points $(\tau_1, \lambda_1), \dots, (\tau_m, \lambda_m)$ are the zeros of $p$ on $\mathbb{T}^2$. Here, $0 \le m \le n$. \\

\item $B_{\alpha}$ is the rational function \[B_{\alpha}(z) := \frac{ \tilde{p}_1(z)-\alpha p_2(z)}{\alpha p_1(z) -\tilde{p}_2(z)},\] where any common factors of the numerator and denominator have been cancelled.  \\

\item $E_{\alpha}$ and $L_k$ are the sets in $\mathbb{T}^2$ defined by $E_\alpha : = \{ (\zeta, \overline{B_\alpha(\zeta)}): \zeta \in \mathbb{T} \}$ and $L_k = \{\tau_k\} \times \mathbb{T}$ for $k=1, \dots, m$. \\

\item $W_\alpha$ is the function on $\mathbb{T}$ defined by \[W_{\alpha}(\zeta):=  \frac{ |p_1(\zeta)|^2 - |p_2(\zeta)|^2}{|\tilde{p}_1(\zeta)- \alpha p_2(\zeta)|^2}.\]
\end{itemize}

Lastly, we say $\alpha \in \mathbb{T}$ is an \emph{exceptional value} for $\phi$ if there is a $k$ such that
 $\phi^*(\tau_k, \lambda_k) = \alpha$ and $\alpha \in \mathbb{T}$ is a \text{generic value} for $\phi$ otherwise. 
\end{definition}

Subsection \ref{sec:clark1} contains our first main result, the following complete characterization of the Clark measures $\sigma_\alpha$ associated to a given degree $(n, 1)$ RIF $\phi$:

 \begin{theorem} \label{thm:clark} For $\alpha \in \mathbb{T}$, the Clark measure $\sigma_\alpha$ satisfies
 \[ \int_{\mathbb{T}^2} f(\zeta) \ d\sigma_{\alpha}(\zeta) = \int_{\mathbb{T}} f (\zeta, \overline{B_\alpha(\zeta)})  \ d \nu_\alpha(\zeta) + \sum_{k=1}^m c_k^\alpha \int_{\mathbb{T}} f(\tau_k, \zeta) \ dm(\zeta)\]
 for all $f \in L^1(\sigma_{\alpha})$,  where  $d\nu_\alpha =W_{\alpha} dm$, the functions $B_{\alpha}, W_\alpha$ are from Definition \ref{def:1},
  and the constants $c_k^\alpha$ are nonzero (and positive) if and only if $\phi^*(\tau_k, \lambda_k) = \alpha$. 
 \end{theorem}
 
When they are non-zero, the constants $c_k^\alpha$ can be obtained from the formula in \eqref{eqn:formc}. We should note that although Clark measures can often be computed in the one-variable case, very is known in the two-variable setting. Indeed, this theorem 
can be viewed as a significant generalization of Example 4.3 in \cite{D19}, which described $\sigma_\alpha$ for the specific degree $(1,1)$ RIF given in \eqref{eqn:ex}, and of Example 3 in \cite{MacD82}, which includes the case where $\phi(z) = z_2 b(z_1)$, for $b$ a finite Blaschke product with $b(0) \in \mathbb{R}$.  Because of its length, the proof of Theorem \ref{thm:clark} is broken into two pieces, Propositions \ref{thm:RIF1n} and  \ref{prop:eclark}.
  
In Section \ref{sec:clark2}, we prove our other main result, a formula for  the isometry $J_\alpha: K_\phi \rightarrow L^2(\sigma_\alpha)$ and an exact characterization of when it is unitary:

\begin{theorem}\label{thm:isom}  Fix $\alpha \in \mathbb{T}$. 
\begin{itemize}
\item[i.] For each each $f \in K_\phi$, $(J_\alpha f ) (\zeta) = f^*(\zeta) \text{ for } \sigma_\alpha\text{-a.e.} \ \zeta\in \mathbb{T}^2.$ 
\item[ii.] $J_{\alpha}: K_{\phi} \rightarrow L^2(\sigma_{\alpha})$ is unitary if and only if $\alpha$ is a generic value for $\phi$.
 
\end{itemize}
 \end{theorem} 

This result is in contrast to the one-variable case, where $J_\alpha$ is always unitary. Here, by computing the non-tangential values of $\phi$ at its finite number of singularities, this theorem allows us to easily identify whether a given $J_\alpha$ is unitary. Part (i) is true in the one-variable setting and follows from a famous (and more general) result of Poltoratski about normalized Cauchy transforms,  see  \cite{Pol93} and \cite[Theorem 10.3.1]{CMR}. Thus, our result can be viewed as a partial two-variable analogue of Poltoratski's result. Again, due to length, we break the proof into two pieces, Propositions \ref{thm:unitary} and  \ref{thm:Jalpha}.

Section \ref{sec:apps} connects our $(n,1)$ results to two related areas of study. First, in Theorem \ref{thm:R}, we connect Theorem \ref{thm:clark} to the theory of Agler decompositions and use our results about $\sigma_\alpha$ and $J_\alpha$ to establish formulas for some of the polynomials $R_j$ in \eqref{eqn:AD2}. Then, we observe that this study of Clark measures can be put into a more general context. Specifically, let $P(\mathbb{T}^2)$ denote the set of Borel probability measures on $\mathbb{T}^2$ equipped with the topology of weak-$\star$ convergence and define 
 \[ \mathcal{P}_2 = \{ f \in \text{Hol}(\mathbb{D}^2): \Re f(z) > 0 \text{ and } f(0, 0) =1\},\]
 which is compact in the topology of uniform convergence on compact subsets of $\mathbb{D}^2$.
Let $M:\mathcal{P}_2 \rightarrow P(\mathbb{T}^2)$ denote the map that takes each $f \in \mathcal{P}_2$ to the unique  
 Borel probability measure $\mu_f$ with 
\[ f(z) = \int_{\mathbb{T}^2} P_z(\zeta) d\mu_f(\zeta) \quad \text{ for } z\in \mathbb{D}^2. \]
Then both $\mathcal{P}_2$ and its image $M(\mathcal{P}_2)$ are compact convex sets and by the Krein-Milman theorem, equal the closed, convex hull of their extreme points. 
It is also easy to show that $f$ is an extreme point of $\mathcal{P}_2$ if and only if $\mu_f$ is an extreme point of $M(\mathcal{P}_2)$. In \cite{Rudin71}, Rudin posed the question
\begin{center} ``What are the extreme points of $\mathcal{P}_2$  (or equivalently, of $M(\mathcal{P}_2)$)?'' \end{center}
While this question is still open, a number of interesting examples and related results (often in the $n$-variable situation) have been proved by Forelli \cite{Forelli81}, Knese \cite{Kne19b}, and McDonald \cite{MacD82, MacD86, MacD87, MacD90}. As the Clark measures $\sigma_\alpha$ are trivially in $M(\mathcal{P}_2)$ when $\phi(0)=0$, it makes sense to consider our investigations in the context of Rudin's question and these subsequent results. In particular, the following  is a quick corollary of Theorem \ref{thm:clark} and a theorem from \cite{Kne19b}:

\begin{corollary}\label{cor:extreme}  Let $\phi =\frac{\tilde{p}}{p}$ be a degree $(n,1)$ RIF with $\tilde{p}(0,0)=0$. If $\alpha \in \mathbb{T}$, then:
\begin{itemize}
\item[i.] If $\alpha$ is an exceptional value for $\phi$, then $\sigma_\alpha$ is not an extreme point of $M(\mathcal{P}_2)$.
\item[ii.] If $p$ is saturated, $\deg p = \deg \tilde{p}$, and $\alpha$ is generic for $\phi$, then $\sigma_\alpha$ is an extreme point of $M(\mathcal{P}_2)$. 
\end{itemize}
\end{corollary}

Part (i) of this corollary shows that in our setting, the more complicated Clark measures cannot be extreme points of $M(\mathcal{P}_2)$. Meanwhile part (ii) coupled with our earlier characterizations provide explicit formulas for some extreme points of $M(\mathcal{P}_2)$. We should mention that these appear somewhat related to the measures studied in \cite[Example 3]{MacD82}.

In the last section, we use our results to compute the Clark measures and study the $J_\alpha$ isometries associated to several degree $(n,1)$ RIFs. First, in Example \ref{ex:fave}, we use our results to recover Example 4.3 in \cite{D19}. 
Then in Example \ref{ex:AMY}, we apply our results to a more complicated degree $(2,1)$ RIF with a single singularity and in Example \ref{ex:3}, we study a degree $(3,1)$ RIF with two singularities. 
Finally, in Example \ref{ex:deg33}, we investigate a degree $(3,3)$ RIF $\phi$ with a singularity at $(1,1)$. In particular, we show that although $\phi^*(1,1) = -1$ and hence $\alpha =-1$ is an exceptional value for $\phi$, the operator $J_{-1}$ is unitary. This demonstrates that, in its current form, Theorem \ref{thm:isom}(ii) does not extend to general RIFs.

\section{Clark measures for RIFs} \label{sec:RIF}
We begin with some remarks about Clark measures associated with general RIFs on the bidisk, before focusing on the degree $(n,1)$ case.  First, the following result appears to be known, see for instance \cite[p.732]{MacD90} and \cite{LNprep}, but here, we give a simple proof in the RIF case by adapting some of the arguments from the one-variable proof of \cite[Theorem 9.2.1]{CMR}.

\begin{theorem} \label{thm:RIF1} If $\phi$ is a nonconstant RIF on $\mathbb{D}^2$ and $\alpha \in \mathbb{T}$, then $\sigma_{\alpha}$ does not possess any point masses.
\end{theorem}

\begin{proof} Without loss of generality, we will show that $\sigma_\alpha$ does not possess a point mass at $(1,1)$. First, applying  \cite[Proposition 2.6]{D19} with $w=(0,0)$ yields
\begin{equation} \label{eqn:pm} \int_{\mathbb{T}^2}  \frac{1}{(1-z_1 \bar{\zeta}_1)(1-z_2 \bar{\zeta}_2)} d \sigma_{\alpha}(\zeta) = \frac{1-\phi(z) \overline{\phi(0,0)}}{(1-\bar{\alpha} \phi(z))(1-\alpha \overline{\phi(0,0)})}.\end{equation}
Observe that $\theta(z) := \phi(z,z)$ is a nonconstant finite Blaschke product.  Then for $0<r<1$, set $z =r$ and multiply both sides of \eqref{eqn:pm} by $(1-r)^2$ to get
\begin{equation} \label{eqn:pm2} \int_{\mathbb{T}^2} \frac{(1-r)^2}{(1-r \bar{\zeta}_1)(1-r \bar{\zeta}_2)} d \sigma_{\alpha}(\zeta) =  \frac{(1-r)^2 (1-\theta(r) \overline{\theta(0)})}{(1-\bar{\alpha} \theta(r))(1-\alpha \overline{\theta(0)})}.\end{equation}
 Observe that 
\[ \lim_{r \nearrow 1} \frac{(1-r)^2}{(1-r \bar{\zeta}_1)(1-r \bar{\zeta}_2)} =\left\{ \begin{array}{cc} 1 & \text{ if } \zeta = (1,1), \\
0 & \text{otherwise.} \end{array} \right.\]
Since $\sigma_\alpha$ is a finite measure, the dominated convergence theorem implies
\[ \lim_{r \nearrow 1} \int_{\mathbb{T}^2}  \frac{(1-r)^2}{(1-r \bar{\zeta}_1)(1-r \bar{\zeta}_2)} d \sigma_{\alpha}(\zeta) 
=\sigma_\alpha\{ (1,1) \}.\]
As $\theta$ is a nonconstant finite Blaschke product, $\theta(1)$ exists and equals some $\lambda \in \mathbb{T}$ and its derivative $\theta'(1)$ exists and is nonzero, see \cite[Lemma 7.5]{GMR2}. If $\lambda \ne \alpha$, then
\[  \lim_{r \nearrow 1}   \frac{(1-r)^2 (1-\theta(r) \overline{\theta(0)})}{(1-\bar{\alpha} \theta(r))(1-\alpha \overline{\theta(0)})} = 
\lim_{r \nearrow 1} (1-r)^2  \frac{1-\lambda \overline{\theta(0)}}{(1-\bar{\alpha} \lambda)(1-\alpha \overline{\theta(0)})} =0.\]
If $\lambda = \alpha$, then 
\[ 
\begin{aligned} \lim_{r \nearrow 1}   \frac{(1-r)^2 (1-\theta(r) \overline{\theta(0)})}{(1-\bar{\alpha} \theta(r))(1-\alpha \overline{\theta(0)})} &=
 \frac{1-\alpha \overline{\theta(0)}}{\bar{\alpha}- \overline{\theta(0)}} \cdot  \lim_{r\nearrow 1} \frac{ 1-r}{\theta(1)- \theta(r)} \cdot \lim_{r\nearrow 1} (1-r) \\
 & =\alpha  \cdot \frac{1}{\theta'(1)} \cdot 0 =0. 
 \end{aligned}\]
 Equating the two sides in \eqref{eqn:pm2} implies that $\sigma_\alpha\{ (1,1) \} =0.$
\end{proof}

We can say a little bit more about Clark measures associated with RIFs. The papers \cite{BPS17,BPSprep} include several results concerning boundary behavior of two-variable RIFs, and in particular, the structure of their unimodular level sets. Recall from \eqref{eqn:Ca1} that for $\alpha \in \mathbb{T}$ and $\phi = \tilde{p}/p$, 
\[\mathcal{C}_{\alpha}=\{\zeta \in \mathbb{T}^2\colon \tilde{p}(\zeta)=\alpha p(\zeta)\}.\]
Then each $\mathcal{C}_{\alpha}$ satisifes $\mathcal{C}_{\alpha}=\{\zeta \in \mathbb{T}^2\colon \phi^*(\zeta)=\alpha\}\cup (\mathcal{Z}_p\cap \mathbb{T}^2)$, and as was shown in \cite[Theorem 2.8]{BPSprep}, the components of $\mathcal{C}_{\alpha}$ can be locally parametrized using one-variable analytic functions. Intuitively speaking, this implies that, for each $\alpha$, the Clark measure $\sigma_{\alpha}$ of any two-variable RIF has support contained in a one-dimensional subset of $\mathbb{T}^2$.  (We should mention that, technically speaking, the result in \cite{BPSprep} was proved for $\phi = \frac{\tilde{p}}{p}$ with $\deg p = \deg \tilde{p}$, but that assumption does not appear to materially affect the conclusions.)  As an aside, we also note that general pluriharmonic measures on $\mathbb{T}^2$ can have substantially larger, and even two-dimensional support, viz. \cite{MacD90}.

It should be noted that knowing that $\sigma_{\alpha}$ is a Clark measure associated with {\it some} RIF and is supported on {\it some} set of the form $\{\zeta \in \mathbb{T}^2\colon \tilde{p}(\zeta)=\alpha p(\zeta)\}$ does not suffice to determine that measure (or its associated RIF) uniquely. Indeed, one can exhibit (see Example \ref{ex:AMY}) two different RIFs $\phi_1=\frac{\tilde{p}_1}{p_1}$ and $\phi_2=\frac{\tilde{p}_2}{p_2}$ whose Clark measures are not multiples of each other but are both supported on the same set $\{\zeta\in \mathbb{T}^2\colon \tilde{p}_1(\zeta)=\alpha p_1(\zeta)\}=\{\zeta \in \mathbb{T}^2\colon \tilde{p}_2(\zeta)=\alpha p_2(\zeta)\}$ for some $\alpha \in \mathbb{T}$. Thus, in order to study Clark measures and isometries for two-variable RIFs, we will need to perform a structural analysis of the measures $\sigma_{\alpha}$ that goes beyond determining their supports.
 
 \section{Clark measures for Degree $(n,1)$ RIFs} \label{sec:n1}

Throughout the rest of this paper, we let $\phi = \tilde{p}/p$ denote a fixed degree $(n,1)$ rational inner function for some $n \ge 1$. Recall that we can decompose $p$ as in \eqref{eqn:p1p2}. Then the polynomials $p$, $\tilde{p}$ share no common factors. Moreover, $p$ has no zeros on $\mathbb{D}^2 \cup (\mathbb{D} \times \mathbb{T}) \cup (\mathbb{T} \times \mathbb{D})$ and at most $n$ distinct zeros on $\mathbb{T}^2.$ See for example, Lemma 10.1, the proof of Corollary 13.5, and Appendix $C$ in \cite{Kne15}. 

\begin{remark} \label{rem:notation} For a given degree $(n,1)$ RIF function $\phi$ and $\alpha \in \mathbb{T}$, recall the objects from Definition \ref{def:1} and define the following additional objects: \\

\begin{itemize}

\item $\nu_\alpha$ is the measure on $\mathbb{T}$ defined by $d \nu_\alpha  := W_\alpha dm$. \\

\item $Q$ and $R_1, \dots, R_n$ are the polynomials given in Theorem \ref{thm:model}. \\

\item $b_1, \dots, b_n$ are the rational functions in the disk algebra $A(\mathbb{D})$ from Proposition \ref{prop:QR}. \\

\item For $f \in K_\phi$, $h, g_1, \dots, g_n$ are the $H^2(\mathbb{D})$ functions given in Theorem \ref{thm:model}.

\end{itemize}
\end{remark} 

Finally, recall that $\alpha \in \mathbb{T}$ is an \emph{exceptional value} for $\phi$ if there is a $k$ such that
 $\phi^*(\tau_k, \lambda_k) = \alpha$ and $\alpha \in \mathbb{T}$ is a \emph{generic value} for $\phi$ otherwise.

\subsection{Model Space Preliminaries}  \label{sec:model}

Clark measures are closely related to the model space $K_\phi$ and so, we pause to record some known facts about $K_\phi$ in the degree $(n,1)$ case.

\begin{theorem} \label{thm:model} There are polynomials $Q, R_1, \dots, R_n \in \mathbb{C}[z_1,z_2]$ such that $\deg R_j \le (n-1,1)$, $\deg Q \le (n,0)$ and for $z,w \in \mathbb{C}^2$, 
\begin{equation} \label{eqn:AD1}  p(z) \overline{p(w)} - \tilde{p}(z) \overline{\tilde{p}(w)} = (1-z_1 \bar{w}_1) \sum_{j=1}^n R_j(z) \overline{R_j(w)}+ (1-z_2\bar{w}_2) Q(z) \overline{Q(w)}. \end{equation}
Furthermore, each $R_j$ and $Q$ vanish at each $(\tau_k,\lambda_k)$ and a function $f \in K_\phi$ if and only if there exist $g_1, \dots, g_n, h \in H^2(\mathbb{D})$ such that 
\begin{equation} \label{eqn:ktheta}  f(z) = \tfrac{Q}{p}(z) h(z_1)+\sum_{j=1}^n \tfrac{R_j}{p}(z) g_j(z_2) \quad \text{ for } z\in \mathbb{D}^2.\end{equation}
Finally, if $f \in K_\phi$ is written as in \eqref{eqn:ktheta}, then
\[ \| f\|_{K_\phi}^2 = \| f\|^2_{H^2(\mathbb{D}^2)} =  \| h\|^2_{H^2(\mathbb{D})}+\sum_{j=1}^n \| g_j\|^2_{H^2(\mathbb{D})}.\]
\end{theorem}

\begin{proof} As this result is not new, we just give some intuition and references for the different components of the theorem. First, note that on $H^2(\mathbb{D}^2),$ there are two shift operators, $M_{z_1}$ and $M_{z_2}$, defined by $(M_{z_i}f)(z) = z_if(z)$ for $i=1,2.$  Let  $S^{\mathrm{max}}_1$ be the maximal $M_{z_1}$-invariant subspace of $K_{\phi}$, where $M_{z_1}$ is multiplication by $z_1$. Then, while not obvious, it is true that $S^{\mathrm{min}}_2 :=K_{\phi} \ominus S^{\mathrm{max}}_1$ is invariant under $M_{z_2}$, see  \cite{bsv05, b12}.  Let $K_1, K_2$ denote the reproducing kernels of the two Hilbert spaces
\[ S^{\mathrm{max}}_1 \ominus M_{z_1}  S^{\mathrm{max}}_1:=\mathcal{H}(K_1)  \ \ \text{ and }   \ \ S^{\mathrm{min}}_2 \ominus M_{z_2}  S^{\mathrm{min}}_2:=\mathcal{H}(K_2),\]
respectively. This yields the Agler decomposition
\[  1-\phi(z) \overline{\phi(w)} = (1-z_1\bar{w}_1) K_2(z,w) + (1-z_2 \bar{w}_2) K_1(z,w).\]
Since $\phi$ is a degree $(n,1)$ RIF, one can show that $\dim \mathcal{H}(K_1)=1$ and $\dim \mathcal{H}(K_2)=n$. Let $Q/p$ be an orthonormal basis for $\mathcal{H}(K_1)$
 and  $R_1/p, \dots, R_n/p$ be an orthonormal basis for $\mathcal{H}(K_2)$.  One can  show that the $Q, R_j$  are polynomials with $\deg Q \le (n,0)$ and  $\deg R_j \le (n-1, 1)$ and each $Q, R_j$ vanishes at each $(\tau_k, \lambda_k)$.  Furthermore, 
\[ K_1(z,w) =  \frac{Q(z)\overline{Q(w)}}{p(z) \overline{p(w)}} \text{ and }   K_2(z,w) = \sum_{j=1}^n  \frac{R_j(z) \overline{R_j(w)}}{p(z) \overline{p(w)}},\]
and substituting the formulas into the Agler decomposition and multiplying through by the denominator gives \eqref{eqn:AD1}.
For the details, see for example \cite{bsv05, bickne13, Kne10} and the references within.

The characterization of functions in $K_\phi$ from \eqref{eqn:ktheta} follows from the fact that the reproducing kernel $k(z,w)$ of $K_\phi$ satisfies
\[ k(z,w)  =  \frac{1}{1-z_2\bar{w}_2} \sum_{j=1}^n  \frac{R_j(z) \overline{R_j(w)}}{p(z) \overline{p(w)}} + \frac{1}{1-z_1\bar{w}_1} \frac{Q(z)\overline{Q(w)}}{p(z) \overline{p(w)}}\] 
and from standard properties of reproducing kernels. A proof of the formula for the norm of functions in $K_\phi$ can be found, for example, in Remark 2.3 in \cite{bg18}.
\end{proof}

\subsection{Support Sets and Consequences} \label{sec:support}

In this subsection, we obtain some information about the objects from Definition \ref{def:1} and Remark \ref{rem:notation}. First, recall that  $\mathcal{C}_\alpha$ from \eqref{eqn:Ca1} contains the support of $\sigma_\alpha.$ This theorem gives additional information about $\mathcal{C}_\alpha$ and $B_\alpha$.

\begin{theorem} \label{thm:Ca} Let $\alpha \in \mathbb{T}$.
\begin{itemize}
\item[i.] Then $B_\alpha$ is a finite Blaschke product.
\item[ii.] Let $\alpha$ be a generic value of $\phi$. Then $\deg B_\alpha = n$ and $\mathcal{C}_\alpha$ equals $E_\alpha$.
\item[iii.] Let $\alpha$ be an exceptional value of $\phi$ and (after reordering if necessary), assume $\phi^*(\tau_k, \lambda_k) =\alpha$ for $k=1,\dots, \ell$. Then $\deg B_\alpha = n-\ell$ and  $\mathcal{C}_\alpha$ is $E_\alpha \cup (\cup_{k=1}^{\ell} L_k)$. 
\end{itemize}
\end{theorem}

\begin{proof} Fix $\alpha \in \mathbb{T}$  and observe that $\mathcal{C}_{\alpha}$ is the set of $\zeta \in \mathbb{T}^2$ satisfying
\[ 
\begin{aligned}
0= \tilde{p}(\zeta) - \alpha p(\zeta) &= \zeta_2(\tilde{p}_1-\alpha p_2)(\zeta_1) + (\tilde{p}_2 -\alpha p_1)(\zeta_1)\\
& = 
(\tilde{p}_1-\alpha p_2)(\zeta_1) \left(\zeta_2 - 1/B_\alpha(\zeta_1) \right).
\end{aligned}
\] 
Too see that $B_\alpha$ is a finite Blaschke product, observe that if
  \[ r = \alpha p_1 -\tilde{p}_2, \text{ then } \tilde{r} = \bar{\alpha} \tilde{p_1} -p_2.\]
Since $|r| = |\tilde{r}|$ on $\mathbb{T}$, this implies $| B_{\alpha}(\zeta)|=1$ on $\mathbb{T}$. 
As $\phi$ is nonconstant, $|\phi(z_1,0)| = |\frac{\tilde{p}_2}{p_1}(z_1)| <1$ on  $\mathbb{D}$ and so,  $\alpha p_1 -\tilde{p}_2$ is nonvanishing on $\mathbb{D}$. This implies $B_\alpha$ is a finite Blaschke product and the common zeros of its original numerator and denominator are exactly the zeros of $\tilde{p}_1-\alpha p_2$ in $\mathbb{T}$. Denote those zeros by $\gamma_1, \dots, \gamma_\ell$. Then
\begin{equation} \label{eqn:Balpha2}  \tilde{p}(z) - \alpha p(z)  = q(z_1)\left(z_2 - 1/B_\alpha(z_1)  \right)\prod_{k=1}^\ell(z_1-\gamma_k),\end{equation}
where $q \in \mathbb{C}[z]$ is the numerator of $B_\alpha$ once common terms have been cancelled. Equation \eqref{eqn:Balpha2} shows  that $\mathcal{C}_{\alpha}$ is the set of lines $\{\gamma_k\} \times \mathbb{T}$ for $k=1, \dots, \ell$ and $E_\alpha.$ By   Lemma \ref{lem:factor}(i), we can further assume that  $\gamma_1,\dots, \gamma_\ell$ are distinct.

To establish $\deg B_\alpha$, note that $\deg (\tilde{p}_1-\alpha p_2) = n$. This occurs because since $|\phi|<1$ on $\mathbb{D}^2,$ we have $|p_1(0)| > |\tilde{p}_2(0)|$ and thus, the coefficient of the degree $n$ term in  $\tilde{p}_1-\alpha p_2$ is nonzero. Thus, $\deg B_\alpha = n-\ell$, its original degree minus the number of cancelled terms or equivalently, the number of lines of the form $\{\gamma\} \times \mathbb{T}$ in $\mathcal{C}_\alpha$. 

For (ii), let $\alpha$ be generic. By Lemma \ref{lem:nt} below, $\mathcal{C}_\alpha$ cannot contain any lines of the form $\{\gamma\} \times \mathbb{T}$ and so, our arguments give $\mathcal{C}_\alpha = E_\alpha$ and $\deg B_\alpha =n$.

For (iii), let $\alpha$ be exceptional and (after reordering if necessary), assume $\phi^*(\tau_k, \lambda_k) =\alpha$ for $k=1,\dots, \ell$. Then  Lemma \ref{lem:nt} implies that $L_1, \dots, L_\ell$ are  exactly the lines of the form $\{\gamma\} \times \mathbb{T}$ in $\mathcal{C}_\alpha$. Then the above arguments imply  $\mathcal{C}_\alpha = E_\alpha \cup (\cup_{k=1}^{\ell} L_k)$ and $\deg B_\alpha = n-\ell.$
\end{proof}

In the above proof, we used the following two lemmas.
 
 \begin{lemma}\label{lem:nt}  For $\gamma \in \mathbb{T}$, the set $\mathcal{C}_{\alpha}$ contains $\{\gamma\}\times \mathbb{T}$ if and only if $\gamma =\tau_k$ for some $k$ and $\phi^*(\tau_k,\lambda_k) = \alpha.$
 \end{lemma}
 
 \begin{proof} Observe that $\{\gamma\}\times \mathbb{T} \subseteq \mathcal{C}_{\alpha}$ if and only if $\phi^*(\gamma, \zeta_2) \equiv \alpha$ for all $\zeta_2 \in \mathbb{T}$, except maybe at one $\zeta_2$ where $p(\gamma,\zeta_2)=0.$

Thus, for the forward direction, we can assume $\phi^*(\gamma, \zeta_2) \equiv \alpha$ (except maybe at one $\zeta_2)$. As our assumptions imply $\deg \tilde{p}(\gamma, \zeta_2)=1$, the polynomials $p(\gamma, \cdot), \tilde{p}(\gamma, \cdot)$ must share a common factor with a zero on $\mathbb{T}$, say $(z_2 - \beta).$ This implies that $p(\gamma, \beta) =\tilde{p}(\gamma, \beta) =0$. Thus, $(\gamma, \beta) = (\tau_k, \lambda_k)$ for some $k.$

Set $\tau:=(\tau_k, \lambda_k)$ and write
\[ p(z) = \sum_{j=M}^{n+1}  P_j(\tau- z) \  \text{ and } \tilde{p}(z) = \sum_{j=M}^{n+1} Q_j(\tau - z),\]
where $P_j$ and $Q_j$ are homogeneous polynomials in $z_1, z_2$ of degree $j$ and $M \ge 1$. Define $\lambda = \phi^*(\tau)$.  By \cite[Propositions 14.3, 14.5]{Kne15}, $P_M = \lambda Q_M$ and since $\tilde{p}, p$ share no common factors, we can conclude that $M=1$ and $P_1$ contains a term $cz_2$ with $c \ne 0$. Then
\[ \alpha \equiv \phi(\tau_k, \zeta_2 ) = \frac{\lambda c (\lambda_k-\zeta_2)}{c(\lambda_k-\zeta_2)} = \lambda,\]
so $\alpha = \phi^*(\tau).$  Similarly, if $\alpha = \phi^*(\tau)$ for some $\tau=(\tau_k,\lambda_k)$, the above equality and  arguments imply that we also have $\alpha \equiv \phi^*(\tau_k, \zeta_2)$ for all $\zeta_2 \in \mathbb{T}\setminus \{\lambda_k\}$  and so $\{\tau_k\}\times \mathbb{T}$ is in $\mathcal{C}_{\alpha}$. 
 \end{proof}
 
 The following lemma describes finer behavior of $B_\alpha$ and $\phi$ related to the singularities $(\tau_k, \lambda_k)$ for $k=1, \dots, m$.
 
     \begin{lemma}\label{lem:factor}  For $\alpha \in \mathbb{T}$,
   \begin{itemize}
   \item[i.] $\tilde{p}- \alpha p$ and $\tilde{p}_1-\alpha p_2$ do not possess repeated linear factors $(z_1-\gamma)^2$ with $\gamma\in \mathbb{T}.$
   \item[ii.] $\overline{B_\alpha(\tau_k)} = \lambda_k$ for $k=1, \dots, m$.
   \item[iii.] If $\mathcal{C}_\alpha$ contains $L_k$, then for all $z_2 \in \overline{\mathbb{D}}$ with $z_2 \ne \lambda_k$, $\frac{\partial \phi}{\partial z_1} (\tau_k,z_2)$ equals a fixed nonzero constant $C$. 
   \end{itemize}
 \end{lemma}
 
 \begin{proof} By the factorization in \eqref{eqn:Balpha2},
 it is easy to see that for $k \in \mathbb{N}$, $(z_1-\gamma)^k$ is a factor of $\tilde{p}- \alpha p$ if and only if it is a factor of $\tilde{p}_1-\alpha p_2.$
 Then, by way of contradiction, assume $\tilde{p}- \alpha p$ is divisible by $(z_1-\gamma)^2$. Then, 
  $\gamma =\tau_k$ for some $k$ and for $\zeta_2 \in \mathbb{T} \setminus \{\lambda_k\}$, we have  $\phi(\gamma, \zeta_2) \equiv \alpha$ and
 \[ \tfrac{\partial \phi}{\partial z_1}(\gamma, \zeta_2) =\frac{ p\frac{ \partial \tilde{p}}{\partial z_1} -\tilde{p} \frac{ \partial p}{\partial z_1}}{p^2}(\gamma, \zeta_2) = \frac{ \frac{ \partial \tilde{p}}{\partial z_1} -\alpha \frac{ \partial p}{\partial z_1}}{p}(\gamma, \zeta_2) =0.\]
 Fix any $\lambda\in \mathbb{T}$ with $\lambda \not \in \{ \lambda_1, \dots, \lambda_m\}$. Then $\phi_\lambda(z_1) : = \phi(z_1, \lambda)$ is a nonconstant finite Blaschke product and so, $\phi'_\lambda(\gamma) \ne 0$. See, for example, Lemma 7.5 in \cite{GMR2}. Since
 \[ 0 \ne \phi'_\lambda (\gamma) =  \tfrac{\partial \phi}{\partial z_1}(\gamma, \lambda) =0\]
by the above argument, we obtain the requisite contradiction. 
 
 For (ii), let $\tau := (\tau_k,\lambda_k)$. First assume that $\phi^*(\tau) \ne \alpha$. As $\tilde{p}(\tau)=p(\tau)=0$, it follows by definition that $\tau \in \mathcal{C}_\alpha$. Thus, Theorem \ref{thm:Ca} implies $\lambda_k = \overline{B_\alpha(\tau_k)}.$ Now if  $\phi^*(\tau) = \alpha$, then as in the proof of Lemma \ref{lem:nt}, we can write:
 \[ (p-\alpha \tilde{p})(z) = \sum_{j=2}^{n+1} (P_j-\alpha Q_j)(z-\tau) = (\tau_k-z_1) G(z),\]
where $Q_j, P_j$ are homogeneous polynomials of degree $j$ and $G$ is a polynomial. Here, we used the fact that $P_1 =\alpha Q_1$ and $\deg (p-\alpha \tilde{p}) \le (n,1)$. From this, it is clear that $G(\tau) =0.$ Using the proof of Theorem \ref{thm:Ca}, we know $(z_1-\tau_k)$ divides $\tilde{p}_1-\alpha p_2$ and so,
\[ (\tau_k-z_1) G(z) =
(\tilde{p}_1-\alpha p_2)(z_1) \left(z_2 - 1/B_\alpha(z_1) \right)= r(z_1) (\tau_k-z_1)  \left(z_2 - 1/B_\alpha(z_1) \right).\]
 for some  $r \in \mathbb{C}[z]$. By (i), $r(\tau_k) \ne 0$. Dividing through by $(\tau_k-z_1)$ and plugging in $\tau$ implies $\lambda_k= 1/B_\alpha(\tau_k) = \overline{B_\alpha(\tau_k)}.$
 
 For (iii), by the assumptions, there is  some  $r \in \mathbb{C}[z]$ with $r(\tau_k) \ne 0$ such that
 \[(\tilde{p}-\alpha p)(z) = (z_1-\tau_k) r(z_1) (z_2- 1/B_\alpha(z_1)). \]
Then for $z_2 \in \mathbb{C} \setminus \{\lambda_k\}$, we have  $\phi(\tau_k, z_2) \equiv \alpha$ and we 
can use (ii) to conclude
 \[ \tfrac{\partial \phi}{\partial z_1}(\tau_k, z_2) =\frac{ \frac{ \partial \tilde{p}}{\partial z_1} -\alpha \frac{ \partial p}{\partial z_1}}{p}(\tau_k, z_2) = \frac{r(\tau_k) (z_2- \lambda_k)}{c(z_2-\lambda_k)} =\frac{r(\tau_k)}{c},\] 
for some $c \ne 0$.
 \end{proof}
  
 We can also use the set $\mathcal{C}_\alpha$ to refine our understanding of the polynomials $R_1, \dots, R_n, Q$ from Theorem \ref{thm:model} as follows:
 \begin{proposition} \label{prop:QR}  $R_1, \dots, R_n, Q$ satisfy the following properties:
 \begin{itemize}
 \item[i.] For $\zeta_1 \in \mathbb{T}$,  $|Q(\zeta_1)|^2 = |p_1(\zeta_1)|^2 -  |p_2(\zeta_1)|^2$.
 \item[ii.] For $z \in \overline{\mathbb{D}}^2$, $R_j(z) = r_j(z_1) \left( 1- B_{\alpha}(z_1)z_2\right) + z_2 Q(z_1) b_j(z_1)$, for some unique $r_j \in \mathbb{C}[z]$ with $\deg r_j \le (n-1)$ and rational $b_j \in A(\mathbb{D}).$   
 \end{itemize}
 \end{proposition}
 
 \begin{proof} Part (i) follows from some algebra; substituting $z=w=(\zeta_1, z_2)$ into \eqref{eqn:AD1} gives
 \[ |p(\zeta_1, z_2)|^2- |\tilde{p}(\zeta_1, z_2)|^2 = (1-|z_2|^2) |Q(\zeta_1)|^2.\]
 The left-hand-side becomes 
 \[ |p_1(\zeta_1)+ z_2p_2(\zeta_1)|^2 - |z_2 \overline{p_1(\zeta_1)} +   \overline{p_2(\zeta_1)}|^2 = (1-|z_2|^2) \left( |p_1(\zeta_1)|^2 -  |p_2(\zeta_1)|^2\right),\]
 and dividing by $(1-|z_2|^2)$ gives the desired formula.
 
 For (ii), recall that $\tilde{p}(z)=\alpha p(z)$ whenever $z_2 = 1/B_\alpha(z_1)$. 
 Substituting $z_2 = 1/B_{\alpha}(z_1)$ and  $w_2 = 1/B_{\alpha}(w_1)$ into \eqref{eqn:AD1} gives
\[ 0 = (1-z_1 \bar{w}_1) \sum_{j=1}^n R_j(z_1, 1/B_{\alpha}(z_1)) \overline{R_j(w_1, 1/B_{\alpha}(w_1))} + (1- 1/B_{\alpha}(z_1) \overline{ 1/B_{\alpha}(w_1)})Q(z_1) \overline{Q(w_1)},\]
for all $z_1, w_1$ where $B_\alpha(z_1) \ne 0$.
We can rewrite this as
\[ B_{\alpha}(z_1) \overline{B_{\alpha}(w_1)} \sum_{j=1}^n R_j(z_1, 1/B_{\alpha}(z_1)) \overline{R_j(w_1, 1/B_{\alpha}(w_1))} = \frac{1- B_{\alpha}(z_1) \overline{B_{\alpha}(w_1)}}{1-z_1 \bar{w}_1} Q(z_1) \overline{Q(w_1)},\]
for $z_1, w_1 \in \mathbb{D}.$
 Then the right-hand-side is the reproducing kernel of the one-variable model space $\hat{K}_{B_\alpha}:=H^2(\mathbb{D}) \ominus B_\alpha H^2(\mathbb{D})$ (which is composed of rational functions in $A(\mathbb{D})$, see  \cite[Chapter 5]{GMR}) times the $Q$ term. To finish the proof, fix $R_j$ and by Theorem \ref{thm:model}, write
  \[ R_j(z) = r_j(z_1) + z_2 q_j(z_1),\] 
  for $r_j , q_j \in \mathbb{C}[z]$. Then standard properties of reproducing kernels imply that for  some  $b_j \in \hat{K}_{B_\alpha}$,
\[ B_\alpha(z_1)  r_j(z_1) +  q_j(z_1) =  b_j(z_1) Q(z_1) \]
for $z_1 \in \overline{\mathbb{D}}.$ Solving this for $q_j$ and substituting back into the formula for $R_j$ yields the desired result. 
 \end{proof}
 
In what follows, we will also require information about the weight function $W$
 \begin{equation} \label{eqn:W} W(x, \zeta) =W_{x}(\zeta)  =\frac{ |p_1(\zeta)|^2 - |p_2(\zeta)|^2}{|\tilde{p}_1(\zeta)- x p_2(\zeta)|^2} \text{ for } (x, \zeta) \in \mathbb{T}^2.\end{equation}
 
 \begin{lemma} \label{lem:Wbd} The function $W$ from \eqref{eqn:W} satisfies the following properties:
\begin{itemize}
\item[i.] $W$ is well defined and continuous on $\mathbb{T}^2$, except possibly at the finite set of points $(\alpha, \tau_k)$ where $\alpha$ is exceptional for $\phi$ and  $L_k \subseteq \mathcal{C}_{\alpha}.$
\item[ii.] For $\alpha \in \mathbb{T}$, $W_\alpha$ has at most a finite number of discontinuities on $\mathbb{T}$, all of which are removable. So, $W_\alpha$ equals a bounded, continuous function $m$-a.e.~on $\mathbb{T}.$
 \end{itemize}
 \end{lemma}
 \begin{proof} For (i), it follows from the proof of Theorem \ref{thm:Ca} that $\tilde{p}_1(\zeta)-x p_2(\zeta)$ only vanishes at some $(\zeta, x) \in \mathbb{T}^2$ if $x$ is exceptional, $\zeta = \tau_k$ for some $k$, and $L_k \subseteq \mathcal{C}_x$. 
 
For (ii), first observe that if $\alpha$ is generic, then (i) implies $W_\alpha$ is continuous, and hence bounded, on $\mathbb{T}$. If $\alpha$ is exceptional, after reordering the singularities of $\phi$, assume $L_1, \dots, L_\ell$ are exactly the lines in $\mathcal{C}_\alpha$. Then by  the proof of Theorem \ref{thm:Ca},
\[   (\tilde{p}_1-\alpha p_2)(z) = q(z) \prod_{k=1}^\ell (z - \tau_k),\]
for some $q \in \mathbb{C}[z]$ that is non-vanishing on $\mathbb{T}$. Similarly, Theorem \ref{thm:model} and Proposition \ref{prop:QR} imply that for $\zeta \in \mathbb{T}$, 
\[ |p_1(\zeta)|^2 - |p_2(\zeta)|^2 = |Q(\zeta)|^2  = \prod_{k=1}^\ell |\zeta - \tau_k|^2 |r(\zeta)|^2,\]
for some $r \in \mathbb{C}[z]$. Thus, for $\zeta \ne \tau_1, \dots \tau_\ell$, 
\[ W_\alpha (\zeta) = \left|\tfrac{r}{q} (\zeta) \right|^2.\]
This shows that the only possible singularities $W_\alpha$ could have on $\mathbb{T}$ are at $\tau_1, \dots, \tau_\ell$ and any such singularities must be removable. \end{proof}
 
 \subsection{Clark Measure Formulas}  \label{sec:clark1}
 
Recall that the Clark measures associated to $\phi$ are characterized via Theorem \ref{thm:clark}. We split the proof into two propositions; the first considers generic values for $\phi$ and the second considers exceptional values for $\phi$.

\begin{proposition} \label{thm:RIF1n}  Let $\alpha \in \mathbb{T}$ be a generic value for $\phi$.
Then for all $f \in L^1(\sigma_{\alpha})$, 
  \begin{equation} \label{eqn:Clark1} \int_{\mathbb{T}^2} f(\zeta) \ d\sigma_{\alpha}(\zeta) = \int_{\mathbb{T}} f (\zeta, \overline{B_\alpha(\zeta)})  \ d \nu_\alpha(\zeta),
  \end{equation}
  where  $d\nu_\alpha =W_{\alpha} dm$ and $B_\alpha, W_\alpha$ are from Definition \ref{def:1}. \end{proposition} 
  
  \begin{proof}  For each fixed $z_2 \in \mathbb{D}$, define the function
\[ \psi^\alpha_{z_2}(z_1):= \frac{\alpha + \phi(z_1,z_2)}{ \alpha - \phi (z_1, z_2)} \text{ for } z_1 \in \mathbb{D}.\]
Recall that $\phi$ has no singularities on $\mathbb{T} \times \mathbb{D}$. Furthermore, since $\alpha$ is generic, if $\phi(\zeta_1, z_2) = \alpha$ for $\zeta_1 \in \mathbb{T}$,  then $z_2 = 1/B_{\alpha}(\zeta_1) \in \mathbb{T}$. 
Thus, $\psi^\alpha_{z_2} \in A(\mathbb{D})$ and for all $z_1 \in \mathbb{D}$,
\[  \Re\left(\psi^{\alpha}_{z_2}(z_1)\right)  = 
\int_{\mathbb{T}} \frac{ 1-|z_1|^2}{|z_1 - \zeta|^2} \Re\left(\psi^\alpha_{z_2}(\zeta)\right) d m(\zeta).  \]
If $z_1 \in \mathbb{D}$ and $\zeta \in \mathbb{T}$, then the computation in the proof of Proposition \ref{prop:QR}(i) gives
\[ |p(\zeta, z_2)|^2 - |\tilde{p}(\zeta, z_2)|^2 =(1-|z_2|^2) \left( |p_1(\zeta)|^2 - |p_2(\zeta)|^2\right),\]
which one can use to obtain
    \[
    \begin{aligned}
  \Re\left(\psi^{\alpha}_{z_2}(\zeta) \right) = \frac{ 1-|\phi(\zeta, z_2)|^2}{|\alpha - \phi(\zeta, z_2)|^2} 
    = \frac{ |p(\zeta, z_2)|^2-|\tilde{p}(\zeta, z_2)|^2}{|\alpha p(\zeta, z_2) -\tilde{p}(\zeta, z_2) |^2} 
    =  \frac{1-|z_2|^2}{ | z_2 - \overline{B_{\alpha}(\zeta)}|^2}  \frac{ |p_1(\zeta)|^2 - |p_2(\zeta)|^2}{ |\alpha p_2(\zeta) -\tilde{p}_1(\zeta)|^2} .
    \end{aligned} 
    \]
  This  implies that 
\[
\begin{aligned}
 \Re\left(\psi^{\alpha}_{z_2}(z_1)\right) 
& = 
 \int_{\mathbb{T}} \frac{ 1-|z_1|^2}{|z_1 - \zeta|^2} \frac{1-|z_2|^2}{| z_2 - \overline{B_{\alpha}(\zeta)}|^2}
   \frac{ |p_1(\zeta)|^2 - |p_2(\zeta)|^2}{ |\alpha p_2(\zeta) -\tilde{p}_1(\zeta)|^2} d m(\zeta) \\
   &= \int_{\mathbb{T}} P_z(\zeta, \overline{B_{\alpha}(\zeta)})  W_\alpha(\zeta) d m(\zeta).
   \end{aligned}
   \]
     Thus, by the definition of $\sigma_\alpha$, we can conclude that  
   \[ \int_{\mathbb{T}^2} P_z(\zeta)  d\sigma_{\alpha}(\zeta) = \int_{\mathbb{T}} P_z(\zeta, \overline{B_{\alpha}(\zeta)}) d \nu_\alpha (\zeta).\]
Since finite linear combinations of Poisson functions $P_z$ are dense in $C(\mathbb{T}^2)$, this formula extends to all functions in  $C(\mathbb{T}^2)$. Since $W_\alpha$ is bounded by Proposition \ref{lem:Wbd}, the formula extends to $f\in L^1(\sigma_\alpha)$ by  Lemma \ref{lem:extension}.\end{proof}
       
Let us now consider the exceptional $\alpha$ values for $\phi$. 
       
 \begin{proposition} \label{prop:eclark} Let $\alpha \in \mathbb{T}$ be an exceptional value for $\phi$ and (after re-ordering the zeros of $p$ on $\mathbb{T}^2$ if necessary and applying Theorem \ref{thm:Ca}), assume $\mathcal{C}_\alpha = E_\alpha \cup (\cup_{k=1}^\ell L_k)$. 
Then for all $f \in L^1(\sigma_{\alpha})$, 
  \begin{equation} \label{eqn:Clark2} \int_{\mathbb{T}^2} f(\zeta) d\sigma_{\alpha}(\zeta) = \int_{\mathbb{T}} f(\zeta,\overline{B_\alpha(\zeta)})  d \nu_\alpha(\zeta) + \sum_{k=1}^\ell c^\alpha_k \int_{\mathbb{T}} f(\tau_k, \zeta) \ dm(\zeta),
  \end{equation}
 where  $d\nu_\alpha =W_{\alpha} dm$, $B_\alpha, W_\alpha$ are from Definition \ref{def:1}, and $c^\alpha_1, \dots, c^\alpha_\ell > 0$. \end{proposition}      
       
\begin{proof} Recall that $\sigma_\alpha$ is supported on $\mathcal{C}_\alpha = E_\alpha \cup (\cup_{k=1}^\ell L_k)$. Since the $L_k$ are disjoint and $E_\alpha \cap L_k =\{ ( \tau_k, \overline{B_\alpha(\tau_k)})\}$, which has $\sigma_\alpha$ measure $0$, we only need to show 
\begin{align}  \label{eqn:sigma1}
\int_{\mathbb{T}^2} f(\zeta) \chi_{E_\alpha}(\zeta) d\sigma_\alpha(\zeta) &=\int_\mathbb{T} f(\zeta, \overline{B_\alpha(\zeta)}) W_\alpha(\zeta) dm(\zeta) \\
\label{eqn:sigma2}
\int_{\mathbb{T}^2} f(\zeta) \chi_{L_k}(\zeta) d\sigma_\alpha(\zeta) &= c^\alpha_k \int_\mathbb{T} f(\tau_k, \zeta) dm(\zeta) \end{align}
for $f \in L^1(\sigma_\alpha)$ and $k=1, \dots, \ell$, where $\chi_E$ denotes the characteristic function of a set $E \subseteq \mathbb{T}^2.$ Then Lemma \ref{lem:Wbd} implies $W_\alpha$ is bounded on $\mathbb{T}$ (with at most $\ell$ removable singularities) and by Lemma \ref{lem:extension}, we need only establish \eqref{eqn:sigma1} and \eqref{eqn:sigma2} for $f \in C(\mathbb{T}^2)$. 

To ease notation, throughout this proof, for any $F$ defined on $\mathbb{T}^2$, we will use $F_\alpha$ to denote $F_\alpha(\zeta) := F(\zeta, \overline{B_\alpha(\zeta)})$ and $F_{k}$ to denote $F_{k}(\zeta):=F(\tau_k,\zeta).$ \\

\noindent \textbf{Part 1.} We first prove \eqref{eqn:sigma1}. To that end, fix a small $\epsilon >0$ and define 
\[S_\epsilon = \{ \zeta \in \mathbb{T}: \min_k| \zeta -\tau_k| < \epsilon\}\]
 and define $S_{\epsilon/2}$ analogously. By Lemma \ref{lem:Wbd} and the definition of $B_\alpha$, we can find a small arc $A_{\alpha} \subseteq \mathbb{T}$ centered at $\alpha$ such that both $W(x,\zeta)$ and $B(x,\zeta):=B_x(\zeta)$ are uniformly continuous on $A_\alpha \times (\mathbb{T} \setminus S_{\epsilon/2}).$ Choose $(\alpha_n) \subseteq \mathbb{T}$ such that each $\alpha_n$ is generic and $(\alpha_n) \rightarrow \alpha$. Then by \cite[Corollary 2.2]{D19}, $(\sigma_{\alpha_n})$ converges weak-$\star$ to $\sigma_\alpha$. 

To exploit that fact, let $\Psi_\epsilon$ be a continuous function on $\mathbb{T}$ such that 
\[ \Psi_\epsilon \equiv 1 \text{ on } \mathbb{T} \setminus S_\epsilon, \ \ \Psi_\epsilon \equiv 0 \text{ on } S_{\epsilon/2}, \ \ 0 \le \Psi_\epsilon \le 1 \text{ on } S_\epsilon \setminus S_{\epsilon/2}.\]
Fix $f \in C(\mathbb{T}^2)$. By our assumptions and by Proposition \ref{thm:RIF1n}, 
\begin{align} 
\int_{\mathbb{T}^2} f(\zeta) \Psi_\epsilon(\zeta_1) d\sigma_\alpha(\zeta) &= \lim_{n\rightarrow \infty} 
\int_{\mathbb{T}^2} f(\zeta) \Psi_\epsilon(\zeta_1) d\sigma_{\alpha_n}(\zeta) \nonumber \\
& =\lim_{n\rightarrow \infty} \int_{\mathbb{T}} f_{\alpha_n}(\zeta) \  \Psi_\epsilon(\zeta) \ d\nu_{\alpha_n}(\zeta)
 = \int_{\mathbb{T}} f_\alpha(\zeta) \  \Psi_\epsilon(\zeta) \  d\nu_\alpha(\zeta), \label{eqn:sigma3}
\end{align}
where the last equality follows because 
\[ 
 \int_{\mathbb{T}}\left|  f_{\alpha_n}(\zeta)  W_{\alpha_n}(\zeta) -f_\alpha(\zeta) W_{\alpha}(\zeta) \right | \Psi_\epsilon(\zeta) dm(\zeta) 
\le \sup_{\zeta \in \mathbb{T} \setminus S_{\epsilon/2}} \left|  \left(f_{\alpha_n}  W_{\alpha_n} -f_\alpha  W_{\alpha}\right)(\zeta) \right | \rightarrow 0,
  \]
 as $n\rightarrow \infty$ because $f_x(\zeta)W(x, \zeta)$ is uniformly continuous on $A_\alpha \times (\mathbb{T } \setminus S_{\epsilon/2}).$ Furthermore, observe that since $\Psi_\epsilon(\zeta_1)\equiv 0$ on each $L_k$, we have
\[ \begin{aligned}  \left| \int_{\mathbb{T}^2} f(\zeta) \chi_{E_\alpha}(\zeta) d \sigma_\alpha(\zeta) -  \int_{\mathbb{T}^2} f(\zeta) \Psi_\epsilon(\zeta_1) d \sigma_\alpha(\zeta) \right| &\le    \int_{\mathbb{T}^2} |f(\zeta)|(1-\Psi_{\epsilon}(\zeta_1)) \chi_{E_\alpha}(\zeta) d \sigma_\alpha(\zeta) \\
&\le  \| f \|_{L^{\infty}(\mathbb{T}^2)} \sigma_\alpha( (S_\epsilon \times \mathbb{T}) \cap E_\alpha).
\end{aligned}
\]
Here, $(S_\epsilon \times \mathbb{T}) \cap E_\alpha$ is the intersection of the curve $E_\alpha$ with thin strips in $\mathbb{T}^2$, see Figure \ref{fig:support}.
\begin{figure} 
\includegraphics[width=0.45 \textwidth]{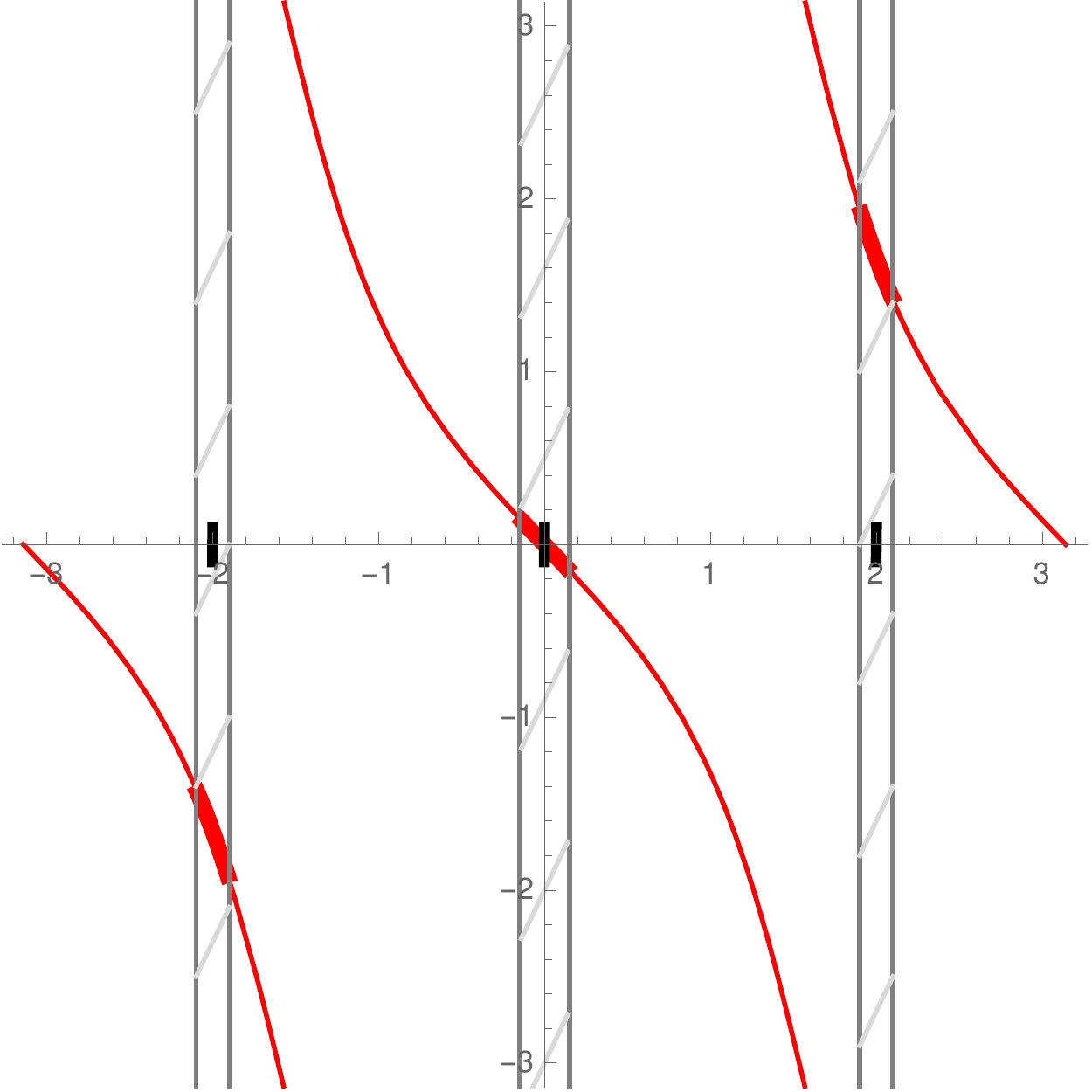}
\caption{\textsl{This shows the set $(S_\epsilon \times \mathbb{T}) \cap E_\alpha$ graphed on $[-\pi,\pi]^2$. $E_\alpha$ is the red curve, $\tau_1, \tau_2, \tau_3$ are the black tics on the horizontal axis, and $S_\epsilon \times \mathbb{T}$ is the union of thin gray strips.}}
\label{fig:support}
\end{figure}
Letting $\epsilon \searrow 0$ and using the dominated convergence theorem gives
\[ \lim_{\epsilon \searrow 0}  \sigma_\alpha( (S_\epsilon \times \mathbb{T}) \cap E_\alpha) = \sigma_\alpha\left( \cup _{k=1}^\ell (\tau_k, \overline{B_\alpha(\tau_k)})\right) =0.\]
This in turn implies that 
\[   \int_{\mathbb{T}^2} f(\zeta) \chi_{E_\alpha}(\zeta) d \sigma_\alpha(\zeta) =\lim_{\epsilon \searrow 0}  \int_{\mathbb{T}^2} f(\zeta) \Psi_\epsilon(\zeta_1) d \sigma_\alpha(\zeta).\] 
As $f_\alpha$ and $W_\alpha$ are both bounded, we can also conclude that 
\[ \lim_{\epsilon \searrow 0}  \int_{\mathbb{T}} f_\alpha(\zeta) \Psi_\epsilon(\zeta) \ d\nu_\alpha(\zeta) = \int_{\mathbb{T}}  f_\alpha(\zeta) \ d \nu_\alpha(\zeta).\]
Combining these last two equalities with \eqref{eqn:sigma3} yields \eqref{eqn:sigma1}.\\

\noindent \textbf{Part 2.} Now we prove \eqref{eqn:sigma2}. We will show that for $z\in \mathbb{D}^2$, \eqref{eqn:sigma2} holds for  $P_z$. Since linear combinations of these are dense in $C(\mathbb{T}^2)$, the result will follow.  
To that end, fix 
 $0<r<1$. Then the definition of $\sigma_\alpha$ gives
\begin{equation} \label{eqn:sigma4}\int_{\mathbb{T}^2} P_{(r \tau_k,z_2)}(\zeta) d\sigma_\alpha(\zeta) =  \Re \left( \frac{\alpha + \phi(r \tau_k,z_2)}{\alpha-\phi(r \tau_k,z_2)}\right). \end{equation}
We will multiply both sides by $(1-r)$ and let $r\nearrow 1$. First, observe that for $\zeta \in \mathbb{T}^2$,
\[ \lim_{r\nearrow 1} (1-r) P_{(r\tau_k,z_2)}(\zeta) = \left\{\begin{array}{cc} 0 & \text{ if } \zeta_1 \ne \tau_k,\\
 2 P_{z_2}(\zeta_2) & \text{ if } \zeta_1 = \tau_k. \end{array} \right.\]
 Then by the dominated convergence theorem,
 \[ \lim_{r\nearrow 1} \int_{\mathbb{T}^2} (1-r)  P_{(r \tau_k,z_2)}(\zeta) d\sigma_\alpha(\zeta) = \int_{\mathbb{T}^2} 2 P_{z_2} (\zeta_2) \chi_{L_k} (\zeta) d\sigma_{\alpha}(\zeta).\]
 Observe that $L_k \subseteq \mathcal{C}_\alpha$ actually implies that $\phi(\tau_k, z_2) = \alpha$ for all $z_2 \in \mathbb{D}$. Furthermore, since $\phi$ is analytic at each $(\tau_k,z_2)$, we have
 \[ \lim_{z_1 \rightarrow \tau_k} \phi(z_1, z_2) =\alpha \text{ and } \lim_{z_1 \rightarrow \tau_k} \frac{ \phi(z_1, z_2)-\alpha}{z_1 -\tau_k} = \tfrac{\partial \phi}{\partial z_1}(\tau_k , z_2) :=C \ne 0,\] 
 by Lemma \ref{lem:factor}. Then Carath\'eodory's theorem, see (VI-3) in \cite{Sar94}, implies 
 \[ \lim_{r \nearrow 1} \frac{ 1-|\phi(r \tau_k,z_2)|}{1-r} = C \tau_k \bar{\alpha} = |C|\] 
 and so
 \[
 \begin{aligned}
\lim_{r\nearrow 1} \Re\left( \frac{(1-r) (\alpha +\phi(r \tau_k, z_2))}{ \alpha-\phi(r \tau_k, z_2)} \right) &=
\lim_{r\nearrow 1} 
(1-r) \frac{ 1-|\phi(r\tau_k,z_2)|^2}{|\alpha - \phi(r\tau_k,z_2)|^2}  \\
& =\lim_{r\nearrow 1} 2 \left | \frac{\tau_k-r\tau_k}{\alpha - \phi(r \tau_k,z_2)}\right|^2  \frac{ 1-|\phi(r \tau_k,z_2)|}{1-r}  = \frac{2}{|C|}.
\end{aligned}
\]
Now set
\begin{equation} \label{eqn:formc} c^\alpha_k=\frac{1}{|C|} = \frac{1}{| \tfrac{\partial \phi}{\partial z_1}(\tau_k , z_2)|} >0.\end{equation}
 Then \eqref{eqn:sigma4} and our subsequent computations combine to give 
 \[   \int_{\mathbb{T}^2} P_{z_2} (\zeta_2) \chi_{L_k} (\zeta) d\sigma_{\alpha}(\zeta) = c^\alpha_k =c^\alpha_k \int_{\mathbb{T}} P_{z_2}(\zeta) dm(\zeta).\]
 Multiplying both sides by $P_{z_1}(\tau_k)$ establishes \eqref{eqn:sigma2} for $f=P_z$ and completes the proof.
\end{proof}              
 The following lemma, which was used in the above proof, is a consequence of standard measure-theory facts. We include its proof here for the ease of the reader. 
 
 \begin{lemma}
 \label{lem:extension}
 Let $\sigma$ be a Borel measure on $\mathbb{T}^2$ and let $\zeta_2 = g(\zeta_1)$ be a continuous curve in $\mathbb{T}^2$. If $W$ is a continuous function defined on $\mathbb{T}$ such that 
 \begin{equation} \label{eqn:sigma} \int_{\mathbb{T}^2} f(\zeta) d\sigma(\zeta) = \int_\mathbb{T} f(\zeta, g(\zeta)) W(\zeta) dm(\zeta) \ \text{ for all } f\in C(\mathbb{T}^2),\end{equation}
 then \eqref{eqn:sigma} holds for all $f \in L^1(\sigma).$
 \end{lemma}
     
\begin{proof} For ease of notation, set $d \nu = W d m$. Then $\nu$ is a Borel measure on $\mathbb{T}$. Furthermore, if $E\subseteq \mathbb{T}^2$ is a Borel set, then 
\begin{equation}\label{eqn:Eg} E_g :=  \{ \zeta_1\in \mathbb{T} : \text{there exists } \zeta_2 \in \mathbb{T} \text{ with } (\zeta_1, \zeta_2) \in E \cap \{ (\zeta, g(\zeta)) : \zeta \in \mathbb{T}\} \}\end{equation}
is the projection of a Borel set in $\mathbb{T}^2$ onto its first coordinate.  This implies that $E_g$ is an analytic set and its characteristic function is Lesbesgue measurable and hence, $\nu$-measurable, see for example Chapter 13 in \cite{Dudley}. We will use $E_g$ frequently because $\chi_E(\zeta, g(\zeta)) = \chi_{E_g}(\zeta)$ for $\zeta \in \mathbb{T}$. 
The following proof has three steps. \\

\noindent \textbf{Step 1:} Establish \eqref{eqn:sigma} for $f=\chi_{U}$,  where $U$ is an arbitrary open set in $\mathbb{T}^2$. Let $\{K_n\}_{n=1}^{\infty}$ be a sequence of nested compact sets with $U=\bigcup_{n}K_n$. Then, by Urysohn's lemma, there exists a sequence $\{f_n\}_{n=1}^{\infty}$ of continuous functions on $\mathbb{T}^2$ having $0\leq f_n\leq 1$ on $\mathbb{T}^2$ and
  \[f_n=1 \quad \textrm{on}\quad K_n \quad \textrm{and} \quad f_n=0 \quad \textrm{on}\quad U^{c}.\]
 Then $\lim_{n\to\infty}f_n(\zeta)=\chi_{U}(\zeta)$ for every $\zeta \in \mathbb{T}^2$, and since $\nu$ is a finite measure, we can apply the dominated convergence theorem to obtain
 \[
 \sigma(U)=\lim_{n\to \infty}\int_{\mathbb{T}^2}f_n(\zeta)d\sigma(\zeta)\\=\lim_{n\to \infty}\int_{\mathbb{T}}f_n(\zeta,g(\zeta))d\nu(\zeta)
 =\int_{\mathbb{T}}\chi_{U}(\zeta,g(\zeta))d\nu(\zeta),
 \]
as desired.\\

\noindent \textbf{Step 2:} Establish \eqref{eqn:sigma} for $f=\chi_{E}$, where $E$ is an arbitrary Borel set in $\mathbb{T}^2$. Since $\sigma$ is a finite Borel measure on $\mathbb{T}^2$, and hence is Radon, for each $n\in \mathbb{N}$ there exists a compact set $K_n$ and an open set $U_n$ such that $K_n \subseteq E \subseteq U_n$ and 
$\sigma(U_n\setminus K_n)<1/n$. 
Urysohn's lemma again guarantees the existence of a sequence $\{f_n\}_{n=1}^{\infty}$ of continuous functions with $0\leq f_n\leq 1$, 
$f_n=1$ on $K_n$, and $f_n=0$ on $U_n^c$. Since
\[\int_{\mathbb{T}^2}|(f_n-\chi_E)(\zeta)| \ d\sigma(\zeta)  \leq \int_{\mathbb{T}^2} \chi_{U_n\setminus K_n}(\zeta) d\sigma(\zeta) =\sigma(U_n\setminus K_n)<\tfrac{1}{n},\]
we have $\| f-f_n\|_{L^1(\sigma)} < 1/n$ and $f_n\to f$ in $L^1(\sigma)$ as $n\to \infty$. Since the $f_n$ are continuous, this implies
\[\int_{\mathbb{T}^2}\chi_E(\zeta)d\sigma (\zeta)=\lim_{n\to \infty}\int_{\mathbb{T}^2}f_n(\zeta) d\sigma(\zeta) =\lim_{n\to \infty}\int_{\mathbb{T}}f_n(\zeta, g(\zeta))d\nu(\zeta).\]
Because  $\chi_{K_n}\leq f_n \leq \chi_{U_n}$ on $\mathbb{T}^2$, we then obtain
\begin{equation}
\int_{\mathbb{T}}\chi_{K_n}(\zeta,g(\zeta))d\nu(\zeta)\leq \int_{\mathbb{T}^2}f_n(\zeta) d\sigma(\zeta) \leq \int_{\mathbb{T}}\chi_{U_n}(\zeta, g(\zeta))d\nu(\zeta)
\label{eqn:sandwich1}
\end{equation}
and as $\chi_{K_n}\leq \chi_E \leq \chi_{U_n}$ on $\mathbb{T}^2$,
\begin{equation}
\int_{\mathbb{T}}\chi_{K_n}(\zeta, g(\zeta))d\nu(\zeta) \leq \int_{\mathbb{T}}\chi_E(\zeta, g(\zeta))d\nu(\zeta) \leq \int_{\mathbb{T}}\chi_{U_n}(\zeta, g(\zeta))d\nu(\zeta).\label{eqn:sandwich2}
\end{equation}
Combining \eqref{eqn:sandwich1} and \eqref{eqn:sandwich2} gives
\[
\begin{aligned}
\left|\int_{\mathbb{T}^2}\chi_{E}(\zeta) d\sigma(\zeta) -\int_{\mathbb{T}}\chi_E(\zeta, g(\zeta))d\nu(\zeta)\right|&\leq  \| \chi_{E} - f_n \|_{L^1(\sigma)}   \\
& +\left |  \int_{\mathbb{T}^2}f_n(\zeta) d\sigma(\zeta) - \int_{\mathbb{T}}\chi_E(\zeta, g(\zeta))d\nu(\zeta)  \right | \\
& \le  \tfrac{1}{n} + \int_{\mathbb{T}}(\chi_{U_n}-\chi_{K_n})(\zeta, g(\zeta)) d\nu(\zeta) \\&= \tfrac{1}{n} + \sigma(U_n\setminus K_n) < \tfrac{2}{n}
\end{aligned}
\]
for all $n$, where we used Step $1$ applied to $U_n\setminus K_n=U_n\cap K_n^c$. Letting $n\rightarrow \infty$ gives \eqref{eqn:sigma} for $f = \chi_E$. \\

\noindent \textbf{Step 3:} Establish \eqref{eqn:sigma} for a general $f\in L^1(\sigma)$. Pick a sequence $\{f_n\}_{n=1}^{\infty}$ of continuous functions on $\mathbb{T}^2$ such that $f_n\to f$ in $L^1(\sigma)$ and pointwise $\sigma$-almost everywhere on $\mathbb{T}^2$. Then there exists some Borel set $E\subset \mathbb{T}^2$ with $\sigma(E)=0$ such that if $f_n(\zeta)\nrightarrow f(\zeta)$, $n\to \infty$, then $\zeta \in E$. Then, if $f_n(\zeta, g(\zeta))\nrightarrow f(\zeta, g(\zeta))$ then $\zeta \in E_g$, where $E_g$ is defined in \eqref{eqn:Eg}. By Step 2, $\nu(E_g)=\sigma(E)=0$. Hence 
\begin{equation}
f_n(\zeta, g(\zeta))\to f(\zeta, g(\zeta))\quad \textrm{for}\quad \nu-\mathrm{a.e}\,\, \zeta \in \mathbb{T}.
\label{eqn:nuaptwise}
\end{equation}

Since the $f_n$ are continuous, we have
\[\int_{\mathbb{T}^2}|f_n(\zeta)-f_m(\zeta) |d\sigma(\zeta) =\int_{\mathbb{T}}|f_n(\zeta, g(\zeta))-f_m(\zeta, g(\zeta))| d\nu (\zeta).\]
This implies that $\{f_n(\zeta, g(\zeta))\}_{n=1}^{\infty}$ is a Cauchy sequence and hence has a limit $F$ in $L^1(\nu)$, and by \eqref{eqn:nuaptwise} we must have $F=f$ in $L^1(\nu)$.  Then $L^1$-convergence gives
\[\int_{\mathbb{T}^2}f(\zeta)d\sigma(\zeta)=\lim_{n\to \infty} \int_{\mathbb{T}^2}f_n(\zeta)d\sigma(\zeta)=\lim_{n\to \infty} \int_{\mathbb{T}}f_n(\zeta, g(\zeta))d\nu(\zeta)=
\int_{\mathbb{T}}f(\zeta, g(\zeta))d\nu(\zeta),\]
which gives \eqref{eqn:sigma} for $f$.
       \end{proof}
       
\subsection{Properties of $J_{\alpha}$} \label{sec:clark2}

Recall that the isometry $J_{\alpha}\colon K_{\phi} \rightarrow L^2(\sigma_{\alpha})$ is obtained by first defining the operator on reproducing kernels $k_w$ as
\[J_{\alpha}[k_w](\zeta):=(1-\alpha\overline{\phi(w)})C_w(\zeta), \quad \text{for }w \in \mathbb{D}^2, \zeta \in \mathbb{T}^2,\]
and then extending it to the rest of $K_\phi$.  Theorem \ref{thm:isom} details our main results about $J_\alpha$, which are proved below in two propositions.

First, unlike the one-variable case, these isometries $J_\alpha$ need not be unitary. The exact situation in our setting is encoded in the following result:
       
 \begin{proposition} \label{thm:unitary} The isometric embedding $J_{\alpha}: K_{\phi} \rightarrow L^2(\sigma_{\alpha})$ is unitary if and only if $\alpha$ is a generic value for $\phi$. 
 \end{proposition}
       
\begin{proof} ($\Rightarrow$) Assume that $\alpha$ is generic.  By Theorem 3.2 in \cite{D19}, we need only show that $A(\mathbb{D}^2)$ is dense in $L^2(\sigma_{\alpha})$. Since $\sigma_{\alpha}$ is a finite Radon measure, $C(\mathbb{T}^2)$ is dense in  $L^2(\sigma_{\alpha})$ and by the Stone-Weierstrass theorem, the set of two-variable trigonometric polynomials is dense in $C(\mathbb{T}^2)$ and hence, in $L^2(\sigma_{\alpha})$. Thus, to show $J_\alpha$ is unitary, we need only show that 
each two-variable trigonometric polynomial agrees with some function in $A(\mathbb{D}^2)$ on $E_\alpha$, which contains the support of $\sigma_\alpha.$ 

Let $h(\zeta) = \zeta_1^m \zeta_2^n$ be an arbitrary trigonometric monomial. To construct a function in $A(\mathbb{D}^2)$ that agrees with $h$ on $E_\alpha$,  first define $t_1, t_2 \in A(\mathbb{D}^2)$ by $t_1(z) = z_1$ and $t_2(z) = z_2$. Then, recall from Theorem \ref{thm:Ca} that $\deg B_{\alpha} =n$. Write $B_{\alpha} = \gamma \prod_{j=1}^n b_{a_j}$, where $\gamma \in \mathbb{T}$ and each $b_{a_j}=(z-a_j)/(1-\bar{a}_j z)$ is the Blaschke factor with zero $a_j \in \mathbb{D}$. Let Let $b^{-1}_{\overline{a}_1}$ denote the inverse function of $b_{\overline{a}_1}$ and define $s_1, s_2 \in A(\mathbb{D}^2)$ by 
\[ s_1(z) =  b_{\overline{a}_1}^{-1} \Big[ \Big( \gamma \prod_{j=2}^n b_{a_j}(z_1) \Big)z_2 \Big],\]
and $s_2(z) = B_\alpha(z_1)$. Then, restricting to $E_\alpha$, we have
\[ s_1(\zeta, \overline{B_{\alpha}(\zeta)}) =b_{\overline{a}_1}^{-1} \Big[ \Big( \gamma \prod_{j=2}^n b_{a_j}(\zeta) \Big)  \overline{ \Big( \gamma \prod_{j=1}^n b_{a_j}(\zeta)\Big)} \Big]  = b_{\overline{a}_1}^{-1} \left[ b_{\overline{a}_1}(\bar{\zeta}) \right] = \bar{\zeta}\]
and $s_2(\zeta, \overline{B_{\alpha}(\zeta)})  = B_\alpha(\zeta).$ As
\[ h(\zeta,  \overline{B_{\alpha}(\zeta)}) = \zeta^m \overline{B_{\alpha}(\zeta)}^n,\]
$h$ agrees with one of $t_1^{|m|}t_2^{|n|},$ $t_1^{|m|}s_2^{|n|},$ $s_1^{|m|} t_2^{|n|},$ $s_1^{|m|}s_2^{|n|}$ on $E_\alpha$. Taking linear combinations of these shows that 
every two-variable trigonometric polynomial agrees with some $F \in A(\mathbb{D}^2)$ on $E_\alpha$ and completes the proof of this forward direction. \\

\noindent ($\Leftarrow$) Assume that $\alpha$ is exceptional and $\phi^*(\tau_k,\lambda_k)=\alpha$. By
way of contradiction, assume that $A(\mathbb{D}^2)$ is dense in $L^2(\sigma_{\alpha})$.  Let $f(\zeta) = \bar{\zeta}_2$.  By assumption, there is a sequence $(f_n) \subseteq A(\mathbb{D}^2)$ that converges to $f$ in $L^2(\sigma_\alpha)$. Then by Theorem \ref{thm:clark}, there is a $c^{\alpha}_k >0$ such that 
\[\int_{\mathbb{T}}| f(\tau_k, \zeta)-f_n(\tau_k, \zeta)|^2dm(\zeta) \le \tfrac{1}{c^{\alpha}_k} \|f-f_n\|^2_{L^2(\sigma_\alpha)} \rightarrow 0,\]
as $n\rightarrow \infty$. Since each $f_n(\tau_k, \cdot)$ is in $H^2(\mathbb{D})$, so is the limit function $f(\tau_k, \cdot)$. Since $f(\tau_k, \zeta) =\bar{\zeta}$, it is clearly not in $H^2(\mathbb{D})$ and so, we obtain the needed contradiction. 
\end{proof}

We can also identify the exact form of the isometric operator $J_\alpha$.

\begin{proposition} \label{thm:Jalpha} For each each $f \in K_\phi$, the isometry $J_\alpha: K_\phi \rightarrow L^2(\sigma_\alpha)$ satisfies 
\[ (J_\alpha f ) (\zeta) = f^*(\zeta) \text{ for } \sigma_\alpha\text{-a.e.} \ \zeta\in \mathbb{T}^2.\] 
\end{proposition}

\begin{proof} Fix $f \in K_\phi$.
By Lemma  \ref{lem:ntvalues},  $f^*$ exists and equals a Borel-measurable function $\sigma_\alpha$-a.e.~on $\mathbb{T}^2$. We claim that $f^* \in L^2(\sigma_\alpha)$ and 
\begin{equation} \label{eqn:sigmabd} \| f^* \|_{L^2(\sigma_\alpha)} \lesssim \| f\|_{K_\phi},\end{equation}
where the implied constant does not depend on $f$. To see this, use Theorem \ref{thm:model} to write
\[  f(z) = \tfrac{Q}{p}(z) h(z_1)+\sum_{j=1}^n \tfrac{R_j}{p}(z) g_j(z_2) \quad \text{ for } z\in \mathbb{D}^2\]
and $g_1, \dots, g_n, h \in H^2(\mathbb{D})$. By the proof of Lemma \ref{lem:ntvalues}, this formula extends to $\mathbb{T}^2$ via non-tangential limits both Lebesgue and $\sigma_\alpha$-a.e. By Proposition \ref{prop:QR}, there is a $b_j \in A(\mathbb{D})$ such that
\[ \left | (\tfrac{R_j}{p})(\zeta, \overline{B_\alpha(\zeta)})\right | = \left | (\tfrac{Q}{p})(\zeta,\overline{B_\alpha(\zeta)}) \right| |b_j(\zeta)|,\]
for all $\zeta \in \mathbb{T}$. Working through the definitions and applying Proposition \ref{prop:QR} give
\[ | (\tfrac{Q}{p})(\zeta,\overline{B_\alpha(\zeta)}) |^2 W_\alpha(\zeta) = \frac{ |p_1(\zeta)|^2-|p_2(\zeta)|^2}{(|p_1(\zeta)|^2-|p_2(\zeta)|^2)^2} |(\tilde{p}_1-\alpha p_2)(\zeta)|^2 \frac{ |p_1(\zeta)|^2 -|p_2(\zeta)|^2}{|(\tilde{p}_1-\alpha p_2)(\zeta)|^2} =1\]
for all $\zeta \in \mathbb{T} \setminus \{\tau_1, \dots, \tau_m\}$. If $B_\alpha$ is non-constant, this immediately implies that 
\[
\begin{aligned}  \int_{\mathbb{T}}& |f^*(\zeta,\overline{B_\alpha(\zeta)})|^2 d \nu_\alpha(\zeta) \\
&\lesssim \int_{\mathbb{T}} | (\tfrac{Q}{p})(\zeta,\overline{B_\alpha(\zeta)}) |^2 \Big( |h^*(\zeta)|^2 + \sum_{j=1}^{n} |b_j(\zeta)|^2 |g^*_j(\overline{B_\alpha(\zeta)}) |^2\Big) W_\alpha(\zeta) dm(\zeta) \\
&=\int_{\mathbb{T}}   |h^*(\zeta)|^2 + \sum_{j=1}^{n} |b_j(\zeta)|^2 |\bar{g}^*_j(B_\alpha(\zeta)) |^2 dm(\zeta) \\
&\lesssim \Big( \| h \|^2_{H^2} + \sum_{j=1}^n \| \overline{g}_j\circ B_\alpha \|_{H^2}^2\Big) \lesssim \| f \|^2_{K_\phi},
\end{aligned}\]
where $\bar{g}$ is the function in $H^2(\mathbb{D})$ whose Taylor coefficients are the complex conjugates of those of $g$.
In this computation, we used Theorem \ref{thm:model} and the well-known fact that composition by a non-constant finite Blaschke product $B_\alpha$ induces a bounded operator on $H^2(\mathbb{D})$, see Theorem 5.1.5 in \cite{MR2007}. Here, the implied constant does not depend on $f$.
If $B_\alpha$ is constant, then the one-variable model space $\hat{K}_{B_{\alpha}} =\{0\}$, so each $b_j \equiv 0$ and  
\[  \int_{\mathbb{T}} |f^*(\zeta,\overline{B_\alpha(\zeta)}))|^2 W_\alpha(\zeta) dm(\zeta) = \| h \|^2_{H^2} \le \| f \|^2_{K_\phi}.\]
 Similarly, for each $1\le j \le n$ and $1 \le k \le m$,  Proposition \ref{prop:QR} gives constants $M_{jk}$ and $d^\alpha_{jk}$ such that for $\zeta \ne \lambda_k$,
\[ \tfrac{Q}{p}(\tau_k, \zeta) \equiv 0 \text{ and } \tfrac{R_j}{p}(\tau_k,\zeta) = M_{jk} \frac{1-B_\alpha(\tau_k) \zeta}{p(\tau_k, \zeta)} =: d^\alpha_{jk},\]
since both the numerator and denominator are linear and by Lemma \ref{lem:factor}, vanish at $\lambda_k$. This shows
\[ c^\alpha_k \int_{\mathbb{T}} |f^*(\tau_k, \zeta)|^2 dm(\zeta) \lesssim \sum_{j=1}^n (d_{jk}^\alpha)^2 \int_{\mathbb{T}} |g_j(\zeta)|^2 dm(\zeta) \lesssim \| f\|_{K_\phi}^2.\]
By Lemma \ref{lem:ntvalues}, this shows $f^* \in L^2(\sigma_\alpha)$. Furthermore, if we define a linear map $T_\alpha: K_\phi \rightarrow L^2(\sigma_\alpha)$ by $(T_\alpha f)= f^*$, then $T_\alpha$ is bounded. Moreover, observe that for  $\zeta \in \mathcal{C}_\alpha \setminus \{ (\tau_1, \lambda_1), \dots, (\tau_m, \lambda_m)\}$, we have
\[ T_\alpha [k_w](\zeta) = (1- \alpha \overline{\phi(w)})C_w(\zeta) = J_\alpha[k_w](\zeta).\]
 Thus, these functions are equal in $ L^2(\sigma_\alpha)$. Since $T_\alpha$ and $J_\alpha$ agree on a dense set of functions in $K_\phi$,  it follows that $T_\alpha = J_\alpha$, which completes the proof. \end{proof}

The proof of Proposition \ref{thm:Jalpha} required the following lemma.

\begin{lemma} \label{lem:ntvalues} If $f \in K_\phi$, then $f^*$ exists and agrees with a Borel measurable function $\sigma_\alpha$-a.e. on $\mathbb{T}^2$ and 
\[  \int_{\mathbb{T}^2} |f^*(\zeta)|^2 d\sigma_\alpha(\zeta) = \int_{\mathbb{T}} |f^*(\zeta, \overline{B_\alpha(\zeta)})|^2 d\nu_\alpha(\zeta) + \sum_{k=1}^m c^\alpha_k \int_{\mathbb{T}} |f^*(\tau_k, \zeta)|^2 dm(\zeta),\]
where $d\nu_\alpha =W_{\alpha} dm$, the functions $B_{\alpha}, W_\alpha$ are from Definition \ref{def:1}, and the $c^\alpha_k$ are from Theorem \ref{thm:clark}.
\end{lemma}

\begin{proof}  By Theorem \ref{thm:model}, there exist $g_1, \dots, g_n, h \in H^2(\mathbb{D})$ such that 
\[  f(z) = \tfrac{Q}{p}(z) h(z_1)+\sum_{j=1}^n \tfrac{R_j}{p}(z) g_j(z_2), \quad \text{ for } z\in \mathbb{D}^2.\]
Let $A \subseteq \mathbb{T}$ be a Borel set with $m(A) = 0$ such that $h, g_1, \dots, g_n$ have non-tangential limits at all $\zeta \in \mathbb{T} \setminus A$.  To finish the set-up, assume $(z^n)=(z_1^n,z_2^n) \rightarrow (\tau_k, \zeta_2) \in \mathbb{T}^2$ non-tangentially, where $\zeta_2 \ne \lambda_k$. Then since $(z_1-\tau_k)$ is a factor of $Q$ and $Q/p$ is continuous near $(\tau_k, \zeta_2)$,
\begin{equation} \label{eqn:Qp} \lim_{n \rightarrow \infty} \left| \tfrac{Q}{p}(z^n) h(z^n_1)\right | \lesssim \lim_{n\rightarrow \infty} |z^n_1-\tau_k| \| h\|_{H^2} \tfrac{1}{\sqrt{1- |z^n_1|^2}} \lesssim  \lim_{n\rightarrow \infty}\sqrt{1-|z^n_1|}=0,\end{equation}
where we also used the reproducing property of $H^2$ and the non-tangential property of $(z^n)$. Similarly, if $B_\alpha$ equals some constant $\gamma \in \mathbb{T},$ then each $\lambda_k = \bar{\gamma}$ and in Proposition \ref{prop:QR}, each $b_j \equiv 0$ and $(z_2- \overline{\gamma})$ divides $R_j$. Thus, arguments analogous to those in \eqref{eqn:Qp} imply that if $(z^n)=(z_1^n,z_2^n) \rightarrow (\zeta_1, \overline{\gamma}) \in \mathbb{T}^2$ non-tangentially with $\zeta_1 \in \mathbb{T} \setminus \{\tau_1,\dots, \tau_m\}$, then
\[ \lim_{n \rightarrow \infty}  \left| \tfrac{R_j}{p}(z^n) g_j(z^n_2)\right | =0, \text{ for }j=1, \dots, n.\]
This implies that $f^*(\zeta)$ exists for all $\zeta \in \mathbb{T}^2 \setminus \hat{A}$, where
\[ \hat{A}:=   \{ (\tau_k, \lambda_k): k=1,\dots, m\} \cup ( (A \setminus \{\tau_1, \dots,\tau_m\}) \times \mathbb{T}) \cup (\mathbb{T} \times A\setminus \{ \bar{\gamma}\} ),\]
where we only include $\bar{\gamma}$ if $B_\alpha$ is constant.
By definition, $\hat{A}$ is a Borel set and we claim $\sigma_\alpha(\hat{A})=0$. Since $\sigma_\alpha$ has no point masses, it is immediate that 
\[ \sigma_\alpha(  \{ (\tau_k, \lambda_k): k=1,\dots, m\}) =0.\]
Set $A_1 = (A \setminus \{\tau_1, \dots,\tau_m\}) \times \mathbb{T}.$ Then as $A_1 \cap L_k =\emptyset$ for each $k$ and  Lemma \ref{lem:Wbd} shows $W_\alpha$ is bounded, we can use Theorem \ref{thm:clark} to compute
\[ \sigma_\alpha(A_1) = \int_{\mathbb{T}}  \chi_{A_1}(\zeta, \overline{B_\alpha(\zeta)}) d\nu_\alpha(\zeta) \lesssim  \int_{\mathbb{T}}  \chi_{A}(\zeta) dm(\zeta) =0.\]
If $B_\alpha$ is non-constant, set $A_2 = \mathbb{T} \times A$. Again by Theorem \ref{thm:clark}, there are constants $c_k^\alpha$ such that 
\[ 
\begin{aligned}
\sigma_\alpha(A_2) &= \int_{\mathbb{T}}  \chi_{A_2}(\zeta, \overline{B_\alpha(\zeta)}) d\nu_\alpha(\zeta) + \sum_{k=1}^m c^\alpha_k \int_{\mathbb{T}} \chi_{A_2} (\tau_k, \zeta) dm(\zeta) \\
&\lesssim   m( \{ \zeta \in \mathbb{T}:  \overline{B_\alpha(\zeta)} \in A\})+ \sum_{k=1}^m c^\alpha_k m(A) =0, 
\end{aligned}\]
where the first set has Lebesgue measure $0$ because non-constant finite Blaschke products are smooth, have non-zero derivatives on $\mathbb{T}$, and are locally invertible on $\mathbb{T}$. Hence, the preimage $\bar{B}_\alpha^{-1}(A)$ must have measure $0$ because $A$ does. If $B_\alpha =\gamma$ is constant, set $A_2 = \mathbb{T} \times (A\setminus \{\bar{\gamma}\}).$ Then
\[ 
\begin{aligned}
\sigma_\alpha(A_2) = \int_{\mathbb{T}}  \chi_{A_2}(\zeta, \bar{\gamma}) d\nu_\alpha(\zeta) + \sum_{k=1}^m c^\alpha_k \int_{\mathbb{T}} \chi_{A_2} (\tau_k, \zeta) dm(\zeta) =0
\end{aligned}\]
by the definition of $A_2$.
Thus, $f^*$ exists $\sigma_\alpha$-a.e.~on $\mathbb{T}^2$. Finally, observe that 
\[ F(\zeta) = \limsup_{r\nearrow 1} \Re(f(r\zeta)) + i  \limsup_{r\nearrow 1} \Im(f(r\zeta)) \]
is Borel measurable since each $f_r(\zeta):=f(r\zeta)$ is continuous on $\mathbb{T}^2$ and $F=f^*$ on $\mathbb{T}^2 \setminus \hat{A}$ and hence $\sigma_\alpha$-a.e. Our prior arguments also imply $f^* (\zeta, \overline{B_\alpha(\zeta)}) =F (\zeta, \overline{B_\alpha(\zeta)})$ for $\nu_\alpha$-a.e.~$\zeta\in \mathbb{T}$ and  $f^* (\tau_k, \zeta) = F(\tau_k,\zeta)$ for $m$-a.e.~$\zeta\in \mathbb{T}.$ To finish the proof, for each $n\in \mathbb{N}$, define the Borel set
\[ 
D_n = \{\zeta \in \mathbb{T}^2: |F(\zeta) |<n\}. 
\]
Then Theorem \ref{thm:clark} combined with the monotone convergence theorem gives
\[
\begin{aligned}
\int_{\mathbb{T}^2} & |f^*(\zeta)|^2  d\sigma_\alpha(\zeta) 
 = \lim_{n \rightarrow \infty}  \int_{\mathbb{T}^2} |F(\zeta)|^2  \chi_{D_n}(\zeta) d\sigma_\alpha(\zeta) \\
& = \lim_{n \rightarrow \infty} \left( \int_{\mathbb{T}} |(F\chi_{ D_n} (\zeta, \overline{B_\alpha(\zeta)})|^2 \ d \nu_\alpha(\zeta) + \sum_{k=1}^m c_k^\alpha \int_{\mathbb{T}} |(F \chi_{ D_n})(\tau_k, \zeta)|^2 dm(\zeta) \right) \\
&  =  \int_{\mathbb{T}} |f^* (\zeta, \overline{B_\alpha(\zeta)})|^2 \ d \nu_\alpha(\zeta) + \sum_{k=1}^m c_k^\alpha \int_{\mathbb{T}} |f^*(\tau_k, \zeta)|^2 dm(\zeta), 
\end{aligned}
\]
which is what we needed to show.
\end{proof}

\section{Applications} \label{sec:apps}
The results from Section \ref{sec:n1} have implications for the structure of Agler decompositions and connections to the study of extreme measures from \cite{Kne19b, MacD82, MacD90} and the references therein. In this section, we again assume $\phi = \frac{\tilde{p}}{p}$ is a degree $(n,1)$ rational inner function and throughout, will use the notation denoted earlier in Definition \ref{def:1} and Remark \ref{rem:notation}.

\subsection{Agler Decompositions}

Recall that each such $\phi$ possesses an Agler decomposition from Theorem \ref{thm:model} arising from a particular orthonormal list in $K_\phi$. Moreover, the polynomial $Q$ in that decomposition can be computed directly on $\mathbb{T}$ via Proposition \ref{prop:QR}. In the case of exceptional $\alpha$, we can apply Theorem \ref{thm:clark} to specify some of the remaining polynomials $R_1, \dots, R_n$ from \eqref{eqn:AD1}.

\begin{theorem} \label{thm:R} Let $\alpha \in \mathbb{T}$ be exceptional for $\phi$ and (after reordering if necessary) assume $\phi^*(\tau_k, \lambda_k) = \alpha$ for $k=1,\dots, \ell$. Using Theorem \ref{thm:Ca}, write $B_\alpha = b^1_\alpha / b^2_\alpha,$ where each $\deg b_\alpha^i = n - \ell.$ Then in \eqref{eqn:AD1}, we can take
\begin{equation} \label{eqn:Rj} R_j(z) = d^{\alpha}_j \Big(  b_\alpha^2(z_1)-z_2 b_\alpha^1(z_1) \Big) \prod_{\substack{1 \le k \le \ell \\ k \ne j}} (z_1 -\tau_k),\end{equation}
for $j =1 , \dots, \ell,$ where each $d^{\alpha}_j>0$ is chosen so $c^\alpha_j \| R_j/p(\tau_j, \cdot) \|^2_{H^2(\mathbb{D})}=1 $ and $c^\alpha_j$ is from Proposition \ref{prop:eclark}.
\end{theorem}

\begin{proof} By the proof of Theorem \ref{thm:model}, the $R_1, \dots, R_n$ from \eqref{eqn:AD1} are exactly obtained by specifying that $R_1/p, \dots, R_n/p$ be an orthonormal basis for $\mathcal{H}(K_2)$. Thus, we need only show that for the $R_j$ defined in \eqref{eqn:Rj}, $R_1/p, \dots, R_\ell/p$ are in $\mathcal{H}(K_2)$ and form an orthonormal set there.

To that end, as in Theorem \ref{thm:model}, let $\hat{R}_1/p, \dots, \hat{R}_n/p$ be some orthonormal basis for $\mathcal{H}(K_2)$. Recall that $\hat{K}_{B_\alpha}:=H^2(\mathbb{D}) \ominus B_{\alpha} H^2(\mathbb{D})$ denotes the one variable model space associated to $B_\alpha.$  Then Proposition \ref{prop:QR} implies that for each $j$, there is a unique polynomial $\hat{r}_j$ with $\deg \hat{r}_j \le n-1$ and function $b_j \in \hat{K}_{B_\alpha}$ such that 
\[ \hat{R}_j(z) = \hat{r}_j(z_1) \Big( 1- B_\alpha(z_1) z_2\Big) + z_2 Q(z_1) b_j(z_1).\]
Define a linear map $T\colon \text{Span}\{ \hat{R}_1, \dots, \hat{R}_n\} \rightarrow \hat{K}_{B_\alpha}$ by first specifying $T(\hat{R}_j) = b_j$ and then extending by linearity.  As $\dim \hat{K}_{B_\alpha} = n-\ell$, it follows that $\dim(\ker T ) \ge \ell$.  If $R \in \ker (T)$, then for some $r$ with $\deg r <n$, 
\begin{equation}\label{eqn:kerR} R(z) = r(z_1) (1- B_\alpha(z_1) z_2) = \frac{r(z_1)}{b^2_\alpha(z_1)}  \Big(  b_\alpha^2(z_1)-z_2 b_\alpha^1(z_1) \Big) = q(z_1) \Big(  b_\alpha^2(z_1)-z_2 b_\alpha^1(z_1) \Big),\end{equation}
where $q \in \mathbb{C}[z]$ with $\deg q < \ell$. Note that the set of such $R$ has dimension $\ell$. By comparing dimensions, each $R$ given in \eqref{eqn:kerR} must be in $\ker(T)$ and hence, each $R$ given in \eqref{eqn:kerR} satifies $R/p \in \mathcal{H}(K_2)$. In particular, this implies that each $R_j$ from \eqref{eqn:Rj} satisfies $R_j/p \in \mathcal{H}(K_2)$.

 To show $R_1/p, \dots, R_\ell/p$ are orthonormal in $K_\phi$, we use Proposition \ref{prop:eclark} and Theorem \ref{thm:Jalpha}. First, observe that those two results combine to imply that $R_j/p(\tau_j, \cdot) \in H^2(\mathbb{D}) \setminus \{0\}$, so $d^\alpha_j$ is well defined. Then, one can use the fact that each $R_j$ vanishes on $E_\alpha$ and each $L_k$ with $1\le k \le \ell$ and $k\ne j$ to conclude:
 \[
 \begin{aligned} 
  \left\langle \tfrac{R_i}{p},  \tfrac{R_j}{p} \right \rangle_{K_\phi} &=   \left\langle J_\alpha \left( \tfrac{R_i}{p} \right),  J_\alpha \left(\tfrac{R_j}{p} \right) \right \rangle_{L^2(\sigma_\alpha)} \\
 &= \int_{\mathbb{T}} \tfrac{R_i}{p}(\zeta,\overline{B_\alpha(\zeta)}) \overline{ \tfrac{R_j}{p}(\zeta,\overline{B_\alpha(\zeta)}) } d \nu_\alpha(\zeta) + \sum_{k=1}^\ell c^\alpha_k \int_{\mathbb{T}}  \tfrac{R_i}{p}(\tau_k, \zeta)\overline{ \tfrac{R_j}{p}(\tau_k, \zeta) } \ dm(\zeta) \\
 & = 0 + \sum_{k = i \text{ or }k= j}   c^\alpha_k \int_{\mathbb{T}}  \tfrac{R_i}{p}(\tau_k, \zeta)\overline{ \tfrac{R_j}{p}(\tau_k, \zeta) } \ dm(\zeta) \\
 & =\left\{ \begin{array}{cc}  1 & \text{ if } i=j \\
0 & \text{ if } i \ne j. \end{array} \right.
 \end{aligned}
 \]
Thus, $\{ R_1/p, \dots, R_\ell/p\}$ is an orthonormal set in $K_\phi$ and hence in $\mathcal{H}(K_2)$, which completes the proof.
\end{proof}

\subsection{Extreme Points} 

Recall that $\mathcal{P}_2 = \{ f \in \text{Hol}(\mathbb{D}^2): \Re f(z) >0 \text{ and } f(0,0) =1\}$ and $M:\mathcal{P}_2 \rightarrow P(\mathbb{T}^2)$ is the map that takes $f \in \mathcal{P}_2$ to the unique  
 Borel probability measure $\mu_f$ on $\mathbb{T}^2$ with 
\[ f(z) = \int_{\mathbb{T}^2} P_z(\zeta) d\mu_f(\zeta) \quad \text{ for } z\in \mathbb{D}^2. \]
for some $f \in \mathcal{P}_2$ and $f$ is an extreme point of $\mathcal{P}_2$ if and only if $\mu_f$ is an extreme point of $M(\mathcal{P}_2)$.  As mentioned in the introduction, Forelli, McDonald, and Knese have proved a number of interesting results related to such extreme points. For example, Knese proved the following result in \cite[Theorem 1.5]{Kne19b}:

\begin{theorem} \label{thm:knese} Let $q$ be a polynomial with no zeros on $\mathbb{D}^2$ and let $\tilde{q}$ be the reflection of $q$ with  $\deg \tilde{q} =\deg q$. Assume that $q$ is $\mathbb{T}^2$-saturated, $\tilde{q}, q$ share no common factors, $\tilde{q}(0,0)=0$, and $q-\tilde{q}$ is irreducible. Then $f : = \frac{q+\tilde{q}}{q-\tilde{q}}$ is an extreme point of $\mathcal{P}_2$. 
\end{theorem} 

As mentioned in the introduction, our results in the $(n,1)$ setting coupled with Theorem \ref{thm:knese} yield Corollary \ref{cor:extreme}, which we restate here for convenience.

\begin{corollary*} \textbf{\ref{cor:extreme}.} Assume $\tilde{p}(0,0)=0$ and let $\alpha \in \mathbb{T}$. Then
\begin{itemize}
\item[i.] If $\alpha$ is an exceptional value for $\phi$, then $\sigma_\alpha$ is not an extreme point of $M(\mathcal{P}_2)$.
\item[ii.] If $\deg p = \deg \tilde{p}$, $p$ is $\mathbb{T}^2$-saturated, and $\alpha$ is generic for $\phi$, then $\sigma_\alpha$ is an extreme point of $M(\mathcal{P}_2)$. 
\end{itemize}
\end{corollary*}

\begin{proof} For (i), without loss of generality, assume $\phi^*(\tau_k, \lambda_k) = \alpha$ for $k=1, \dots, \ell$. By Proposition \ref{prop:eclark}, we can write 
\[ \sigma_\alpha(\zeta) =  \mu_\alpha(\zeta)+ c^\alpha_1 \left( \delta_{\tau_1} (\zeta_1) \otimes m(\zeta_2) \right),\]
 for a positive Borel measure $\mu_\alpha$ on $\mathbb{T}^2$ and $c^\alpha_1 >0$. As $\phi(0,0)=0$,  we have
 \[ 1 = \sigma_{\alpha}(\mathbb{T}^2) = \mu_\alpha(\mathbb{T}^2) + c_1^{\alpha},\]
 and as $\mu_\alpha(\mathbb{T}^2) >0$, we have $c_1^\alpha <1$. 
Then $\hat{\mu}_\alpha: = \frac{1}{1-c^{\alpha}_1} \mu_\alpha$ is a probability measure  and 
\begin{equation} \label{eqn:convex} \sigma_\alpha(\zeta) = (1-c^{\alpha}_1) \hat{\mu}_\alpha(\zeta) +  c^\alpha_1 \left( \delta_{\tau_1} (\zeta_1) \otimes m(\zeta_2) \right),\end{equation}
so $\sigma_\alpha$ is a convex combination of two probability measures on $\mathbb{T}^2$. Clearly, the second one is in $M(\mathcal{P}_2)$, as 
\[ \Re\left( \frac{\tau_1 +z_1}{\tau_1-z_1}\right) = \frac{1-|z_1|^2}{|z_1-\tau_1|^2} = \int_{\mathbb{T}^2} P_z(\zeta) \ d\left( \delta_{\tau_1} (\zeta_1) \otimes m(\zeta_2) \right).\]
 For the first, observe that for each $z \in \mathbb{D}^2$, 
\[ \tfrac{1}{1-c^{\alpha}_1} \Re \left( \frac{\alpha +\phi(z)}{\alpha - \phi(z)} -c^{\alpha}_1 \frac{\tau_1 + z_1}{\tau_1 -
 z_1} \right) = \int_{\mathbb{T}^2} P_z(\zeta) d\hat{\mu}_\alpha(\zeta) > 0.\] 
This implies that $\hat{\mu}_\alpha \in M(\mathcal{P}_2)$ and by \eqref{eqn:convex}, $\sigma_\alpha$ is not  an extreme point in $M(\mathcal{P}_2)$. 

For (ii), choose $\lambda \in \mathbb{T}$ with $\lambda^2 = \alpha$, define $ q = \lambda p$, and set 
\[ f: = \frac{ \alpha +\phi}{\alpha-\phi} = \frac{\alpha p + \tilde{p}}{\alpha p -\tilde{p}} = \frac{ q + \tilde{q}}{q-\tilde{q}}.\]
Note that $\tilde{q} - q = \bar{\lambda}( \tilde{p} -\alpha p)$ must be irreducible by the characterization of $C_\alpha$ from Theorem \ref{thm:Ca}. By Theorem \ref{thm:knese}, $f$ is an extreme point of $\mathcal{P}_2$ and so $\sigma_\alpha$ from Theorem \ref{thm:clark} is extreme in $M(\mathcal{P}_2).$
\end{proof}

 \section{Examples}  \label{sec:examples}
We illustrate our results by examining some specific RIFs and their associated Clark measures in detail. For the first example, we can confirm our general findings at exceptional values $\alpha$ via direct computation.
    \begin{example}\label{ex:fave}
     Let \[\phi(z) = \frac{\tilde{p}(z)}{p(z)} = \frac{2z_1z_2-z_1-z_2}{2-z_1-z_2},\] essentially the example considered in \cite{D19}.  We have the sums of squares decomposition
    \[ |p(z)|^2 - |\tilde{p}(z)|^2 = (1-|z_2|^2) 2|1-z_1|^2 + (1-|z_1|^2) 2 |1-z_2|^2\]
    and for each $\alpha \in \mathbb{T}$, the associated $B_{\alpha}$ is 
    \[ B_{\alpha}(z_1) = \frac{2z_1-1+\alpha} {2\alpha -\alpha z_1 + z_1}.\]
    Note that if $\alpha = -1$, then $B_{-1} \equiv 1$. If $\alpha \ne -1$, then $2\alpha -\alpha z_1 + z_1$ does not vanish on $\mathbb{T}$. Thus if $\alpha \ne -1$, then by Proposition \ref{thm:RIF1n},
 for all $f \in L^2(\sigma_{\alpha})$, we have
\[\int_{\mathbb{T}^2} f(\zeta) d\sigma_{\alpha}(\zeta) = \int_{\mathbb{T}} f(\zeta,\overline{B_\alpha(\zeta)}) \frac{2 |1-\zeta|^2}{|2\zeta-1+\alpha|^2} \ dm(\zeta),\]
and by Theorem \ref{thm:unitary}, the isometric embedding $J_{\alpha}: K_{\phi} \rightarrow L^2(\sigma_{\alpha})$ is unitary. Finally, if $\alpha =-1$, then $|2z_1-1+\alpha|^2 = 4 |z_1 - 1 |^2$. 
By the given sums of squares decomposition, 
\[
\begin{aligned}
\frac{1-|\phi(z)|^2}{| \alpha - \phi(z) |^2}  &= \frac{ |p(z)|^2-|\tilde{p}(z)|^2}{|\alpha p(z) -\tilde{p}(z) |^2} \\
    & = \frac{  |p(z)|^2-|\tilde{p}(z)|^2}{|z_2-1|^2 |-p_2(z_1) -\tilde{p}_1(z_1)|^2 }\\ 
    & = \frac{(1-|z_2|^2) 2|1-z_1|^2 + (1-|z_1|^2) 2 |1-z_2|^2}{|z_2-1|^2 \cdot 4 |z_1-1|^2} \\
& = \frac{1}{2} \left(  \frac{1-|z_2|^2}{|z_2-1|^2}  +  \frac{1-|z_1|^2}{ |z_1-1|^2} \right),
\end{aligned}
\]
which shows $\sigma_{\alpha} = \frac{1}{2}( \delta_1(\zeta_1) \otimes  m(\zeta_2) + m(
\zeta_1) \otimes \delta_1(\zeta_2)).$ This was observed in \cite{D19}, and confirms the contents of Theorem \ref{thm:clark}. Note in particular that
$\frac{\partial \phi}{\partial z_1}(z_1,z_2)=-2\frac{(z_2-1)^2}{(2-z_1-z_2)^2}$, so that $\frac{\partial \phi}{\partial z_1}(1,z_2)=-2$ independent of $z_2$.

See Figure \ref{levelcurves}(a) for a visual representation of the sets $\mathcal{C}_{\alpha}$. $\hfill \blacklozenge$
    \end{example} 
    
Now let us consider a RIF that was not studied in \cite{D19}, and again illustrate how the exceptional measure $\sigma_{\alpha}$ can be identified using both our results and concrete Agler decompositions.

\begin{example}\label{ex:AMY}
Let $\phi = \frac{\tilde{p}}{p}$, where 
\[ p(z) = 4-z_2-3z_1 -z_1z_2+z_1^2 \text{ and } \tilde{p}(z) = 4z_1^2z_2 -z_1^2 -3z_1z_2 -z_1 + z_2.\]
This example was introduced by Agler-M\McC Carthy-Young in \cite{AMY12}. In \cite[Section 15]{Kne15}, Knese provides the following sums of squares decomposition:
\[ |p(z)|^2-|\tilde{p}(z)|^2 = 4(1-|z_2|^2) |1-z_1|^4 +4(1-|z_1|^2) \left( |1-z_1|^2|1-z_2|^2 +2|1-z_1z_2|^2\right).\]
The only singularity of $\phi$ occurs at $(1,1)$. For each $\alpha \in \mathbb{T},$ setting $\phi(z) = \alpha$ and solving for $z_2$ yields $z_2 = 1/B_{\alpha}(z_1)$, where
 \[ B_{\alpha}(z_1) = \frac{4z_1^2 -3z_1+1+\alpha +\alpha z_1}{4\alpha-3z_1\alpha+z_1^2\alpha+z_1^2+z_1}.\]
 As $\phi$ has only one singularity, by previous discussions, the denominator of $B_{\alpha}$ can vanish at a point on $\mathbb{T}$ for at most one $\alpha$. This occurs at $\alpha = -1$, where $B_{-1}$ reduces to $B_{-1}(z) = z$ and $\phi = -1$ has the additional solution $z_1=1$. Thus, we can apply Proposition \ref{thm:RIF1n} to $\alpha \ne -1$ to obtain: 
for all $f \in L^1(\sigma_\alpha)$, 
\[ \int_{\mathbb{T}^2} f(\zeta) d\sigma_{\alpha}(\zeta) = \int_{\mathbb{T}} f(\zeta,\overline{B_\alpha(\zeta)}) \ \frac{ 4|\zeta-1|^4}{|4\zeta^2-3\zeta + 1 + \alpha + \alpha \zeta |^2} \ dm(\zeta).\]
By Theorem \ref{thm:unitary}, the isometric embedding $J_{\alpha}: K_{\phi} \rightarrow L^2(\sigma_{\alpha})$ is unitary for every $\alpha \ne -1$.

\begin{figure}[h!]
    \subfigure[Level curves for $\phi=(2z_1z_2-z_1-z_2)/(2-z_1-z_2)$ corresponding to $\alpha=1$ (black), $\alpha=e^{i\pi/4}$ (gray), $\alpha=e^{i\pi/2}$ (orange), and $\alpha=e^{3i\pi/4}$ (pink). Level set corresponding to exceptional value $\alpha=-1$ marked in red.]
      {\includegraphics[width=0.4 \textwidth]{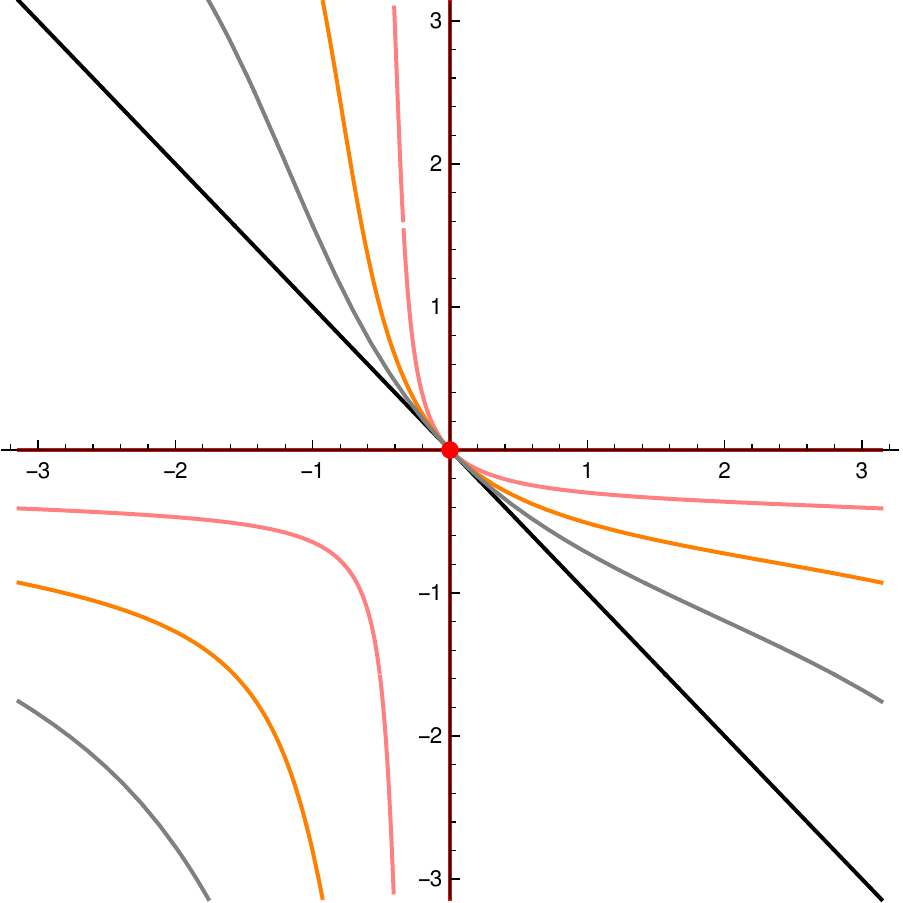}}
    \hfill
    \subfigure[Level curves for $\phi=(4z_1^2z_2 -z_1^2 -3z_1z_2 -z_1 + z_2)/(4-z_2-3z_1 -z_1z_2+z_1^2)$ corresponding to $\alpha=1$ (black), $\alpha=e^{i\pi/4}$ (gray), $\alpha=e^{i\pi/2}$ (orange), and $\alpha=e^{3i\pi/4}$ (pink). Level set corresponding to exceptional value $\alpha=-1$ marked in red.]
      {\includegraphics[width=0.4 \textwidth]{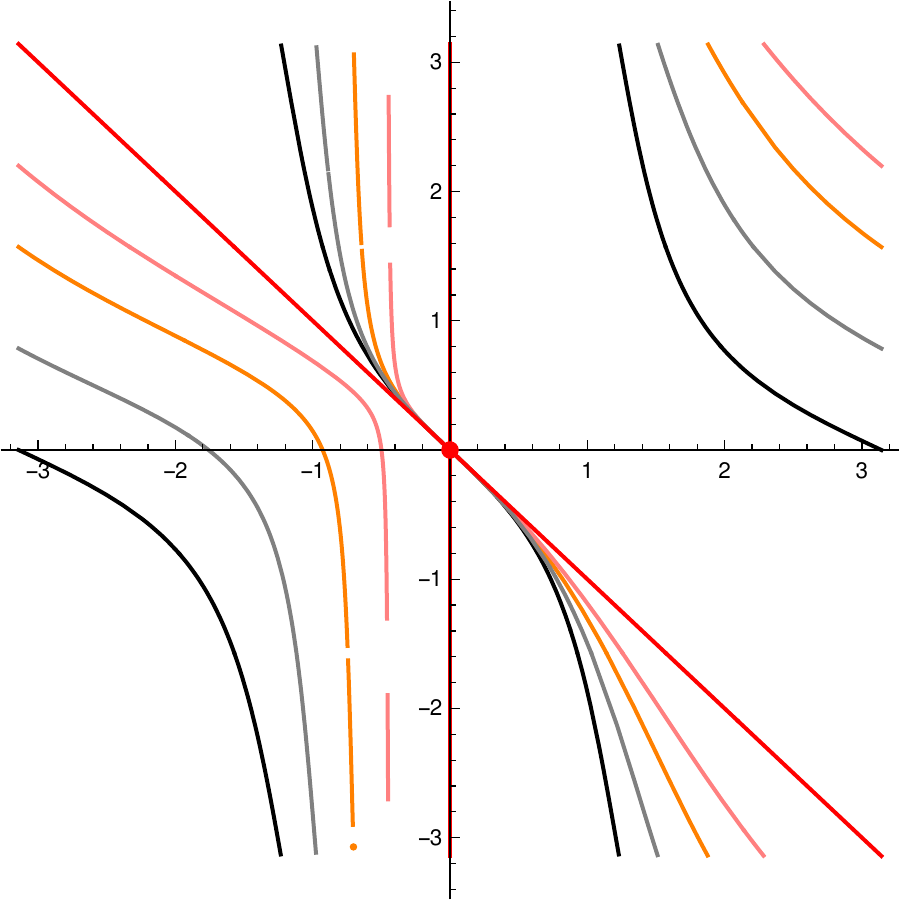}}
  \caption{\textsl{Supports of $\sigma_{\alpha}$ for two different RIFs.}}
  \label{levelcurves}
\end{figure}
Let us now examine the exceptional value $\alpha=-1=\phi^*(1,1)$. A computation shows
\[p(z)+\tilde{p}(z)=4(1-z_1)(1-z_1z_2),\]
and hence, for $\alpha=-1$, we have
\[\frac{1-|\phi(z)|^2}{|\alpha-\phi(z)|^2}=\frac{|p(z)|^2-|\tilde{p}(z)|^2}{16|1-z_1|^2|1-z_1z_2|^2}.\]
By the sums of squares formula above, we then obtain
\[
\begin{aligned}
\frac{1-|\phi(z)|^2}{| \alpha - \phi(z) |^2}  =&\frac{1}{4}|1-z_1|^2\frac{1-|z_2|^2}{|1-z_1z_2|^2}+\frac{1}{4}|1-z_2|^2\frac{1-|z_1|^2}{|1-z_1z_2|^2}+
\frac{1}{2}\frac{1-|z_1|^2}{|1-z_1|^2}\\
=&\frac{1}{4}\frac{(1-|z_1|^2)|1-z_2|^2+(1-|z_2|^2)|1-z_1|^2}{|1-z_1z_2|^2}
\\&+
\frac{1}{2}\frac{1-|z_1|^2}{|1-z_1|^2}
\end{aligned}.
\]

The second term is evidently the Poisson integral of the measure $\sigma_{-1}^{(2)}=\frac{1}{2}(\delta_1(\zeta_1)\otimes m_1(\zeta_2))$, which matches what we get from computing $\frac{\partial \phi}{\partial z_1}(1,z_2)=-2$ and taking the reciprocal of its absolute value.

The first term arises from the measure $\sigma_{-1}^{(1)}$ on $\mathbb{T}^2$ 
having 
\[\int_{\mathbb{T}^2}f(\zeta)d\sigma_{-1}^{(1)}(\zeta)= \tfrac{1}{4}\int_{\mathbb{T}}f(\zeta, \overline{\zeta})|1-\zeta|^2dm(\zeta),\]
as can be seen by examining the Fourier coefficients
\[\widehat{\sigma_{-1}^{(1)}}(k,l)=\left\{\begin{array}{cc}\frac{1}{2}, &k=l\\ 
-\frac{1}{4}, &k=l+1\\-\frac{1}{4}, & l=k+1\\
0 & \mathrm{otherwise}\end{array}\right.\]
and computing the Poisson integral of $\sigma_{-1}^{(1)}$ explicitly. The specific form of $\sigma_{-1}^{(1)}$ of course agrees with Theorem \ref{thm:clark} once we set  $\alpha=-1$ to get $W_{-1}(\zeta)= \frac{4|\zeta-1|^2}{|4\zeta^2-4\zeta|^2}=\frac{1}{4}|\zeta-1|^2$. 

Level curves $\mathcal{C}_{\alpha}$ for several values of $\alpha$ are displayed in Figure \ref{levelcurves}(b).$\hfill \blacklozenge$
 \end{example} 
 \begin{remark}
 The RIF $\phi=\frac{\tilde{p}}{p}$ with
 \[p(z)=2-z_1z_2-z_1^2z_2 \quad \textrm{and}\quad \tilde{p}(z)=2z_1^2z_2-z_1-1\]
 has a singularity at $(1,1)$, and $\phi^*(1,1)=-1$ so that $\alpha=-1$ is an exceptional value. 
 
 One verifies that the associated $B_{-1}(z_1)=z_1$ so that $\sigma_{-1}$ for this example is supported on the same set as the exceptional Clark measure in Example \ref{ex:AMY}. However, we have
 \[W_{\alpha}(\zeta)=\frac{|\zeta-1|^2}{|(2+\alpha)\zeta+\alpha|^2},\]
 which collapses to $W_{-1}(\zeta)=1$ at the exceptional value, meaning that the two Clark measures do not coincide.
  \end{remark}

Our next example is a degree $(3,1)$ RIF with two different singularities on $\mathbb{T}^2$. Here, we are able to observe qualitative differences in $W_{\alpha}$ for the two corresponding exceptional values of $\alpha$ that reflect the finer distinctions between the two singularities.
\begin{example} \label{ex:3}
Let 
\[p(z)=4-z_2+z_1z_2-3z_ 1^2z_2-z_1^3z_2 \quad \textrm{and}\quad \tilde{p}(z)=4z_ 1^3z_ 2-z_1^3+z_1^2-3z_1-1\]
and set $\phi=\frac{\tilde{p}}{p}$. This function has singularities at $(1,1)$ and $(-1,1)$, and the associated exceptional $\alpha$-values are
$\phi^*(1,1)=-1$ and $\phi^*(-1,1)=1$. Level sets for this example are displayed in Figure \ref{multiplelevellines}; see also  \cite[Example 7.4]{BPSprep}.
\begin{figure}
\includegraphics[width=0.4 \textwidth]{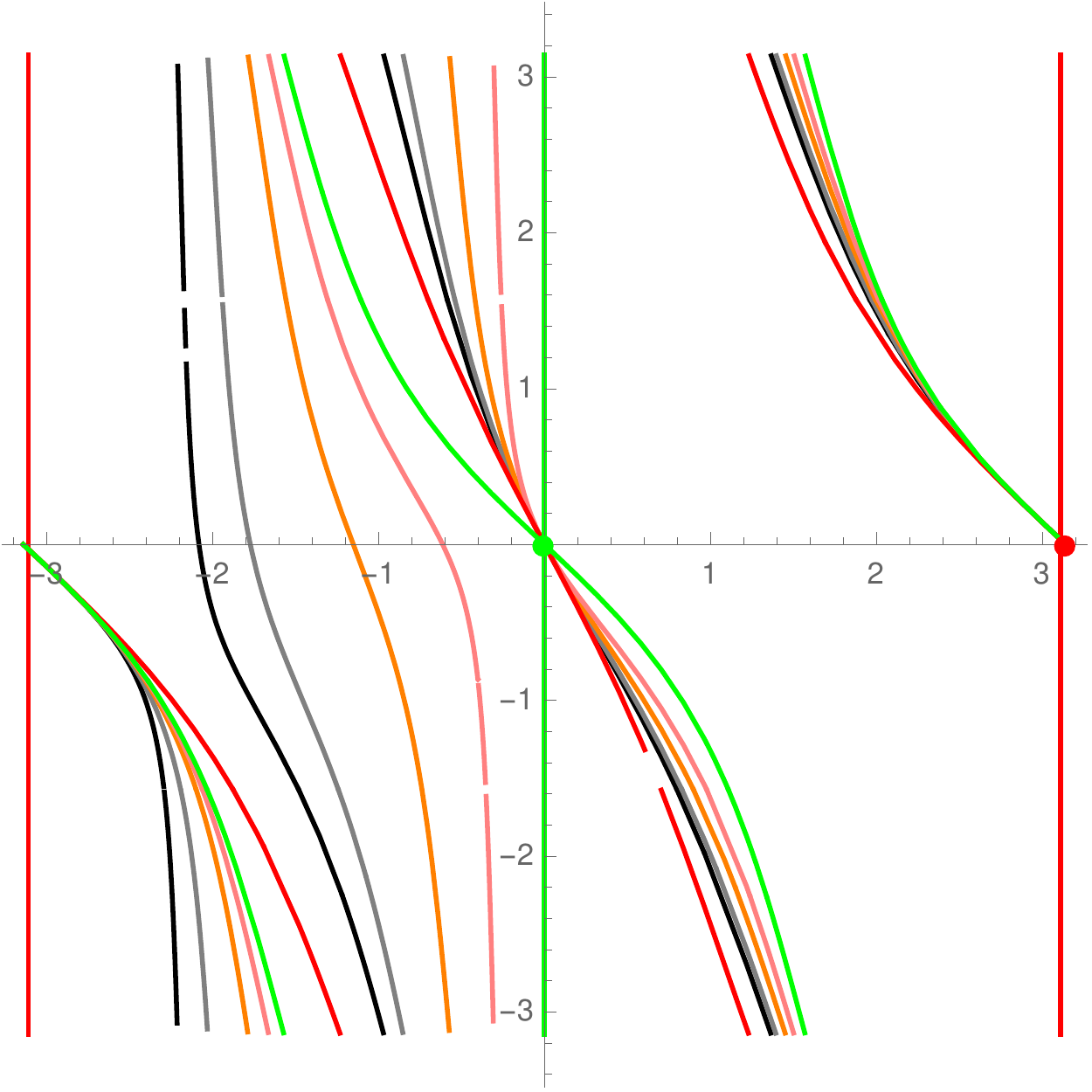}
\caption{\textsl{Generic level curves for $(4z_ 1^3z_ 2-z_1^3+z_1^2-3z_1-1)/(4-z_2+z_1z_2-3z_ 1^2z_2-z_1^3z_2)$ corresponding to several values of $\alpha$ (black, gray, orange, pink). Level sets corresponding to exceptional values $\alpha=-1$ and $\alpha=1$ marked in green and red, respectively.}}
\label{multiplelevellines}
\end{figure}
\end{example}

For $\alpha\neq 1,-1$, we have
\[ B_{\alpha}(z_1)=\frac{\alpha- \alpha z_1 + 3 \alpha z_1^2 + 4 z_1^3 + \alpha z_1^3}{1 + 4 \alpha + 3 z_1 - z_1^2 + z_1^3}.\]
Note that for $\alpha=-1$, we get
\[B_{-1}(z_1)=\frac{z_1-1}{z_1-1}\frac{3z_1^2+1}{3+z_1^2}=\frac{3z_1^2+1}{3+z_1^2},\]
a Blaschke product of degree $2$, while for $\alpha=1$,
\[B_{1}(z_1)=\frac{z_1+1}{z_1+1}\frac{5z_1^2-2z_1+1}{z_1^2-2z_1+5}=\frac{5z_1^2-2z_1+1}{z_1^2-2z_1+5},\]
another degree $2$ Blaschke product. The graphs $\{(\zeta, \overline{B_{-1}(\zeta)})$ and $\{(\zeta, \overline{B_{1}(\zeta)})\}$ together with vertical lines at $\zeta_1=1$ and $\zeta_1=-1$ constitute $\mathrm{supp}(\sigma_{-1})$ and $\mathrm{supp}(\sigma_1)$, respectively.

We further read off that
\[p_1(z_1)=4 \quad \mathrm{and}\quad p_2(z_1)=-1+z_1-3z_1^2-z_1^3\]
so that, with $W_{\alpha}$ as in Remark \ref{rem:notation},
\[W_{\alpha}(\zeta)=\frac{16-|1-\zeta+3\zeta^2+\zeta^3|^2}{|4\zeta^3+\alpha \zeta^3+3\alpha \zeta^2-\alpha \zeta +\alpha|^2}.\]
After some simplifications, we find that
\[W_{\alpha}(z)=\frac{|\zeta-1|^2|\zeta+1|^4}{|4\zeta^3+\alpha \zeta^3+3\alpha \zeta^2-\alpha \zeta +\alpha|^2}.\]
For the exceptional values $\alpha=\pm 1$, the weights in the point mass parts of $\sigma_{\pm1}$ can be obtained by computing
\[\frac{\partial \phi}{\partial z_1}(1,z_2)=-1 \quad \textrm{and} \quad  \frac{\partial \phi}{\partial z_1}(-1,z_2)=-2,\]
which imply
\[c^{-1}_1= \frac{1}{ |\frac{\partial \phi}{\partial z_1}(1,z_2)|}=1 \quad \textrm{and}\quad c^{1}_1=\frac{1}{ |\frac{\partial \phi}{\partial z_1}(-1,z_2) |}=\frac{1}{2}.\]
(Note that $\phi(0,0)=-\frac{1}{4}$ so that the Clark measures $\sigma_{\pm1}$ are not probability measures in this example.)
Putting $\alpha=\pm 1$ in $W_{\alpha}$, we have cancellation in numerator and denominator, and we obtain
\[W_{-1}(\zeta)=\frac{|\zeta+1|^4}{|3\zeta^2+1|^2}\]
and 
\[W_{1}(\zeta)=\frac{|\zeta-1|^2|\zeta+1|^2}{|5\zeta^2-2\zeta^2+1|^2}.\]
This gives us a complete description of the exceptional Clark measures. 

Furthermore, observe that, $W_{-1}(1)\neq 0$ and so, $W_{-1}$ does not vanish at the $z_1$-coordinate of the singularity with non-tangential value $-1.$ In contrast, $W_{1}(-1)= 0$, so $W_{1}$ does vanish at the $z_1$-coordinate of the singularity with non-tangential value $1.$ This mirrors the singular behavior in Example \ref{ex:AMY}, where function $W_{-1}$ vanishes at $\zeta=1$, the $z_1$-coordinate of the singularity where $\phi^*(1,1)=-1$.  This pattern suggests a connection with contact order, which was studied in \cite{BPS17} and governs the integrability of partial derivatives of a RIF $\phi$; in that sense, higher contact order indicates a stronger singularity. In our computations, the singularities at $(1,1)$ in Example \ref{ex:AMY} and at $(-1,1)$ in this example (where the exceptional $W_\alpha$ vanish at the $z_1$-coordinate of the associated singularity) are instances of singularities where $\phi$ exhibits contact order $4$; the singularities in Example \ref{ex:fave} and at $(1,1)$ in this example  (where the exceptional $W_\alpha$ do not vanish at the $z_1$-coordinate of the associated singularity) are singularities where $\phi$ exhibits contact order $2$, the lowest possible contact order.  $\hfill \blacklozenge$

\begin{remark}
It would interesting to investigate how the exact nature of a singularity $\tau \in \mathbb{T}^2$ (contact order, number of branches of $p$ coming together at $\tau$, etc) of a RIF is reflected in the associated exceptional Clark measure. For example, if $\phi=\frac{\tilde{p}}{p}$ is a general degree $(m,n)$ RIF having contact order at least $4$, does the corresponding exceptional Clark measure have a density along $\mathcal{C}_{\alpha}$ that vanishes at $\tau$?
\end{remark}

Our final example is a rational inner function having bidegree $(3,3)$, and is not covered by our general results. It serves as a counterexample showing that Theorem \ref{thm:unitary} fails for higher-degree RIFs, and illustrates some complexities that arise from the fact that for RIFs of bidegree $(m,n)$ with $m,n\geq 2$, a general $\alpha$-level set is not necessarily parametrized by a single function. 
\begin{example}\label{ex:deg33}
Let $\phi(z)=\frac{\tilde{p}}{p}(z)$ where
\[p(z)=2-z_1^2z_2-z_1z_2^2 \quad \textrm{and} \quad \tilde{p}(z)=z_1z_2(2z_1^2z_2^2-z_1-z_2).\]
This example is obtained by applying the level line embedding construction described in \cite[Section 6.1]{BPSprep} to the essentially $\mathbb{T}^2$-symmetric polynomial
\[r(z)=(1-z_1^2z_2)(1-z_1z_2^2).\]

As is guaranteed by the embedding construction, we have $p(1,1)=0=\tilde{p}(1,1)$ and $\phi^*(1,1)=-1$, as well as
\[\tilde{p}+p=2(1-z_1^2z_2)(1-z_1z_2^2).\]
These facts can also be checked directly. We also note that $p$ and $\tilde{p}$, and hence $\phi$, are invariant under the simultaneous coordinatewise rotations $z_j\mapsto e^{2i\pi/3}z_j$ and $z_j\mapsto e^{-2i\pi/3}z_j$. Some level sets of $\phi$ are displayed in Figure \ref{higherdeglevellines}.

Recall from Example \ref{ex:fave} that $|2-x-y|^2-|2xy-x-y|^2=(1-|x|^2)2|1-y|^2+(1-|y|^2)2|1-x|^2$. Substituting $x=z_1^2z_2$ and $y=z_1z_2^2$ into this formula, we get the decomposition
\[|2-z_1^2z_2-z_1z_2^2|^2-|2z_1^3z_2^3-z_1^2z_2-z_1z_2^2|^2=(1-|z_1^2z_2|^2)2|1-z_1z_2^2|^2+(1-|z_1z_2^2|^2)2|1-z_1^2z_2|.\]
It follows that
\[\frac{1-|\phi(z)|^2}{|1+\phi(z)|^2}=\frac{1}{2}\frac{1-|z_1^2z_2|^2}{|1-z_1^2z_2|^2}+\frac{1}{2}\frac{1-|z_1z_2^2|^2}{|1-z_1z_2^2|^2},\]
and by computing Fourier coefficients, one can show that the two expressions on the right are the Poisson integrals of the measures having
\[\int_{\mathbb{T}^2}f(\zeta)d\sigma_{-1}^{(1)}=\int_{\mathbb{T}}f(\zeta, \overline{\zeta}^2)dm(\zeta)\quad \textrm{and}\quad \int_{\mathbb{T}^2}f(\zeta)d\sigma_{-1}^{(2)}=\int_{\mathbb{T}}f(\overline{\zeta}^2, \zeta)dm(\zeta)\]
respectively. By Doubtsov's Theorem $3.2$ in \cite{D19}, which applies to general RIFs, $J_{-1}$ is unitary if and only if the bidisk algebra is dense in $L^2(\sigma_{-1})$.

\begin{figure}
\includegraphics[width=0.4 \textwidth]{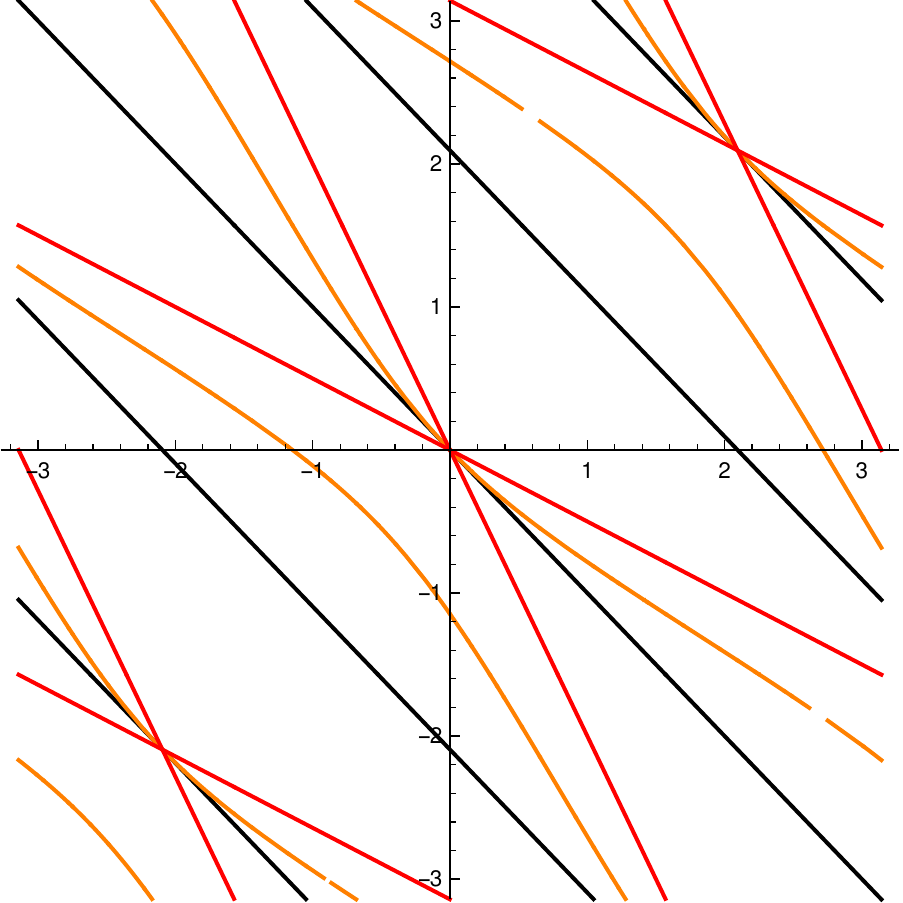}
\caption{\textsl{Level curves for $(2z_1^3z_2^3-z_1^2z_2-z_1z_2^2)/(2-z_1^2z_2-z_1z_2^2)$ corresponding to $\alpha=1$ (black) and $\alpha=e^{\pi i/2}$ (orange). Level set corresponding to exceptional value $\alpha=-1$ marked in red.}}
\label{higherdeglevellines}
\end{figure}
Let us show that this is indeed the case. By definition, $h_1(z)=z_1$ and $h_2(z)=z_2$ are elements of $A(\mathbb{D}^2)$. Next consider the function $g_1(z)=\bar{z}_1$ and the function
\[f_1(z)=z_1z_2+(1-z_1^2z_2)z_2^2 \in A(\mathbb{D}^2).\]
Since
\[f_1(\zeta, \bar{\zeta}^2)=\zeta\bar{\zeta}^2+(1-\bar{\zeta}^2\zeta^2)\bar{\zeta}^4=\bar{\zeta}=g_1(\zeta, \bar{\zeta}^2)\]
and
\[f_1(\bar{\zeta}^2, \zeta)=\bar{\zeta}^2\zeta+(1-\zeta \bar{\zeta}^4)\zeta^2= \zeta^2=g_1(\bar{\zeta}^2,\zeta)\]
we have $g_1=f_1$ on the support of $\sigma_{-1}$. A similar computation shows that the bidisk algebra function
\[f_2(z)=z_1z_2+(1-z_1z_2^2)z_1^2\]
coincides with $g_2(z)=\bar{z}_2$ on $\mathrm{supp}(\sigma_{-1})$. Thus, if $g(\zeta)=\zeta_1^m\zeta_2^n$ is any trigonometric polynomial, then, on the support of $\sigma_{-1}$, $g$ coincides with one of functions $h_1^{|m|}h_2^{|n|}$,  $h_1^{|m|}f_2^{|n|}$, $f_1^{|m|}h_2^{|n|}$, and $f_1^{|m|}f_ 2^{|n|}$, which are all in $A(\mathbb{D}^2)$. Since the trigonometric polynomials are dense in $C(\mathbb{T}^2)$, which in turn is dense in $L^2(\sigma_{-1}$), $A(\mathbb{D}^2)$ is also dense. Thus, $J_{-1}$ is unitary even though $\alpha=-1$ is the non-tangential value of $\phi$ at a singularity. $\hfill \blacklozenge$


\end{example}

\end{document}